\documentclass[12pt]{amsart}
\usepackage[colorlinks=true, pdfstartview=FitV, linkcolor=blue, citecolor=blue, urlcolor=blue, breaklinks=true]{hyperref}

\usepackage{amsmath,amsfonts,amssymb,amsthm,amscd,comment,euscript,stmaryrd}
\usepackage[usenames]{color}
\usepackage[all]{xy}
\usepackage{fancyhdr}
\usepackage{booktabs}
\usepackage{paralist}

\newcommand\Z{\mathbb{Z}}      
\newcommand\N{\mathbb{N}}      
\newcommand\Q{\mathbb{Q}}      
\newcommand\C{\mathbb{C}}      
\newcommand\LL{\mathbb{L}}     
\newcommand\Lw{\Lambda}        

\newcommand\bH{\mathbf{H}}      
\newcommand{\bD}{\mathbf{D}}   
\newcommand\FA{R\llbracket\Lw\rrbracket_F}  
\newcommand{\hh}{\mathtt{h}}      

\DeclareMathOperator{\End}{End}
\DeclareMathOperator{\Spec}{Spec}

\newcommand{\ch}{c^\hh}             
\newcommand{\cc}{\mathfrak{c}}   
\newcommand{\fgl}{{FGL}}         
\newcommand{\CH}{\mathrm{CH}}    
\newcommand{\iDelta}{\mathit{\Delta}} 

\newcommand{\ts}{\textstyle}

\theoremstyle{plain}
\newtheorem{theo}{Theorem}[section]
\newtheorem{prop}[theo]{Proposition}
\newtheorem{lem}[theo]{Lemma}

\theoremstyle{definition}

\newtheorem{defin}[theo]{Definition}
\newtheorem{rem}[theo]{Remark}

\newtheorem{example}[theo]{Example}

\numberwithin{equation}{section}

\begin{document}
%

\title[Formal Hecke algebras and algebraic oriented cohomology theories]{Formal Hecke algebras \\ and algebraic oriented cohomology theories}

\author{Alex Hoffnung}
\author{Jos\'e Malag\'on-L\'opez}
\author{Alistair Savage}
\address{Department of Mathematics and Statistics, University of Ottawa}
\author{Kirill Zainoulline}

\thanks{The work of the second two authors was supported by Discovery Grants from the Natural Sciences and Engineering Research Council of Canada.  The first two authors were supported by the Discovery Grants of the last two.  The first author was also partially supported by funds from the Centre de Recherches Math\'ematiques and the last author was also supported by an Early Researcher Award from the Government of Ontario.}




\begin{abstract}
  In the present paper we generalize the construction of the nil Hecke ring of Kostant-Kumar to the context of an arbitrary formal group law, in particular, to an arbitrary algebraic oriented cohomology theory of Levine-Morel and Panin-Smirnov (e.g.\ to Chow groups, Grothendieck's $K_0$, connective $K$-theory, elliptic cohomology, and algebraic cobordism).  The resulting object, which we call a \emph{formal (affine) Demazure algebra}, is parameterized by a one-dimensional commutative formal group law and has the following important property: specialization to the additive and multiplicative periodic formal group laws yields completions of the nil Hecke and the 0-Hecke rings respectively.  We also introduce a \emph{formal (affine) Hecke algebra}. We show that the specialization of the formal (affine) Hecke algebra to the additive and multiplicative periodic formal group laws gives completions of the degenerate (affine) Hecke algebra and the usual (affine) Hecke algebra respectively.  We show that all formal affine Demazure algebras (and all formal affine Hecke algebras) become isomorphic over certain coefficient rings, proving an analogue of a result of Lusztig.
\end{abstract}

\maketitle \thispagestyle{empty}

\vspace{-0.5cm}

\tableofcontents

\vspace{-0.5cm}

\textbf{Note:} This version of the paper incorporates an erratum to the published version.  The original published version contained a few sign errors in Proposition~\ref{prop:FDA-braid-relations}\eqref{prop-item:Delta-G2-adjacent}.  We thank Marc-Antoine Leclerc for bringing this to our attention.

%
\section{Introduction}
%

Geometric realizations of representations of algebras such as quantized enveloping algebras of Lie algebras and Hecke-type algebras have proved to be an exceptionally interesting and useful tool in both representation theory and geometry.  In particular, the field of geometric representation theory has produced such results as the proof of the Kazhdan-Lusztig conjecture and the construction of canonical bases in quantized enveloping algebras.  Geometric realizations are also often a precursor to \emph{categorification}, a current topic of great interest.

Two fundamental constructions in geometric representation theory are of particular relevance to the current paper.  The first arises from so-called \emph{push-pull operators} (coming from the projection from the flag variety $G/B$ to the quotient $G/P$ of $G$ by a minimal parabolic) on the singular cohomology or $K$-theory (i.e.\ Grothendieck's $K_0$) of the flag variety.  If one works with singular cohomology, these operators generate the nil Hecke algebra.  (When we use the term ``nil Hecke algebra'' here, we do not include the polynomial part.)  If one works instead with $K$-theory, the push-pull operators generate the $0$-Hecke algebra (the specialization of the Hecke algebra at $q=0$).  Adding the operators corresponding to multiplication by elements of the singular cohomology or $K$-theory, one obtains the affine analogues of the algebras above.

The above-mentioned algebras can also be realized in a more algebraic manner.  Let $W$ be the Weyl group of a reduced root system, acting on the weight lattice $\Lambda$. In \cite{KoKu86}, Kostant and Kumar introduced a \emph{twisted group algebra} $Q_W$, which is the smash product of the group ring $\Z[W]$ and the field of fractions $Q$ of the polynomial ring $S$ in $\Lambda$.  Then they defined a subring $R$ of $Q_W$ generated by \emph{Demazure elements} and elements of $S$ and showed that $R$ is similar to the $0$-Hecke algebra: it satisfies the classical braid relation, but a nilpotence relation instead of an idempotence one.  For this reason, they called $R$ the \emph{nil Hecke algebra}.  Following this approach, Evens and Bressler in \cite{EvBr87} introduced the notion of a \emph{generalized Hecke ring} (where the nilpotence/idempotence relation is replaced by a general quadratic one) which includes both $0$-Hecke and nil Hecke algebras as examples.  The Demazure elements play the role of a Hecke basis and have several geometric interpretations (as Demazure operators and push-pull operators) on the singular cohomology of the variety of Borel subgroups associated to the root system.

The second geometric construction relevant to the current paper is the realization of Hecke-type algebras via the geometry of the Steinberg variety.  There is a natural structure of an algebra on the (co)homology of the Steinberg variety via convolution.  Again, the resulting algebra depends on the choice of (co)homology theory.  Equivariant $K$-theory yields the affine Hecke algebra, equivariant singular cohomology yields the degenerate affine Hecke algebra, and top degree Borel-Moore homology yields the group algebra of the Weyl group.  We refer the reader to \cite{CG10,Ginzburg98} and the references therein for further details.

The idea of the current paper is based on the observation that, for the most part, the above-mentioned constructions use only a few properties of $K$-theory and singular cohomology.  These properties can be
summarized by the notion of an algebraic oriented cohomology theory (AOCT).  Roughly speaking, such a theory is a functor $\hh$ endowed with characteristic classes which satisfies the projective bundle formula (see
\cite{LM07}).  They can be classified using the theory of \emph{formal group laws} (FGLs), with the link being provided by the Quillen formula expressing the first characteristic class $\ch_1$ of a tensor product of two line bundles,
\[
  \ch_1(L_1 \otimes L_2)=F(\ch_1(L_1),\ch_1(L_2)),
\]
where $F$ is the one-dimensional commutative FGL associated to $\hh$.  In particular, $K$-theory corresponds to the so-called \emph{multiplicative} FGL $F(u,v)=u+v-uv$ and singular cohomology to the \emph{additive} FGL $F(u,v)=u+v$.  (Note that this correspondence does not, in general, work the other way---there are examples of FGLs that do not correspond to AOCTs.)  Therefore, it is natural to ask whether one can extend the Kostant-Kumar construction of the nil Hecke algebra and the convolution construction of the affine Hecke algebra to the setting of an arbitrary FGL $F$ (and, in particular, for AOCTs).  In the present paper, we provide an affirmative answer to the first question.  We also define algebras that we believe should be related to the more general convolution algebras of the second question.

Given an FGL $F$ (say, corresponding to some AOCT), we introduce the notion of a twisted formal group algebra $Q_W^F$.  To do this, we replace the polynomial ring $S$ of \cite[\S4]{KoKu86} by the formal group algebra associated to $F$.  We then define the \emph{formal Demazure element} to be the expression in $Q_W^F$ corresponding to the formal Demazure operator.  One of our key objects is the algebra generated by the formal Demazure elements and the elements of the formal group algebra.  We call this the \emph{formal affine Demazure algebra} and denote it $\bD_F$.  The subalgebra generated by only the formal Demazure elements is called the \emph{formal Demazure algebra}.  Next, we modify these algebras by introducing an infinite cyclic group.  Geometrically, this corresponds to introducing $\C^*$-actions on the relevant varieties.  We call the resulting algebra the \emph{formal (affine) Hecke algebra} associated to the FGL.  Specializing to the additive and multiplicative periodic FGLs, which correspond to (equivariant) singular cohomology and $K$-theory respectively, we recover (completions of) all of the algebras mentioned above.  This is summarized in the following table.

\medskip

\begin{center}
  \begin{tabular}{lll}
    & Additive \fgl & Multiplicative \fgl \\
    \cmidrule(rl){2-2} \cmidrule(rl){3-3} Alg.\ Oriented Cohom.\ Theory & (Equiv.) singular cohomology & (Equiv.) $K$-theory \\
    Formal Demazure alg.\ & Nil Hecke alg.\ & 0-Hecke alg.\ \\
    Formal affine Demazure alg.\ & Affine nil Hecke alg.\ & Affine 0-Hecke alg.\ \\
    Formal Hecke alg.\ & Group alg.\ of the Weyl Group & Hecke alg.\ \\
    Formal affine Hecke alg.\ & Degenerate affine Hecke alg.\ & Affine Hecke alg.\ \\
  \end{tabular}
\end{center}

\medskip

We see that $\bD_F$ shares many properties with affine Hecke algebras.  However, it does \emph{not} always satisfy the braid relations.  In general, the braid relations are satisfied only up to lower order terms (see Proposition~\ref{prop:FDA-braid-relations}).  This reflects the fact that formal Demazure operators for a general AOCT depend on a choice of reduced decomposition of an element of the Weyl group.  Elliptic versions of the affine Hecke algebras have been studied from the topological point of view by Ginzburg-Kapranov-Vasserot in \cite{GKV95,GKV97}.  However, that setting did not seem to be amenable to explicit computations (contrary to the algebraic setting of the current paper).

Our construction provides two things.  First, it gives a uniform presentation of the fundamental algebras appearing in both the push-pull and Steinberg variety constructions.  Second, it generalizes to other formal groups laws and algebraic oriented cohomology theories, yielding new algebras in the process.  These new algebras should be thought of as natural generalizations of the Hecke-type algebras appearing in the table above.  Given the representation theoretic importance of these Hecke-type algebras, we expect the new algebras defined here to be of interest to both geometers and representation theorists.  In fact, in the subsequent paper \cite{CZZ09}, the authors have related the dual of the formal affine Demazure algebra defined in the current paper to the equivariant oriented cohomology of the flag variety and to the invariants of the formal group algebra under the action of the Weyl group.  Since Hecke-type algebras have played crucial roles in the categorification of quantum groups and related algebras, it is also natural to ask if the generalizations defined in the current paper can be used as building blocks in more general categorifications.

This paper is organized as follows.  In the first four sections, we recall basic definitions and facts used in the rest of the paper. We review the definition of a formal group law and the exponential map in Section~\ref{sec:formal-group-rings}.  In Section~\ref{sec:FGAs}, we recall the definition and basic properties of formal group rings/algebras following \cite[\S2]{CPZ09}. In Section~\ref{sec:formal-dem-ops}, following \cite[\S3]{CPZ09}, we recall the definition and basic properties of formal Demazure operators.  Section \ref{sec:AOCT} is devoted to algebraic oriented cohomology theories.  We define the formal (affine) Demazure algebras and prove various facts about them in Sections~\ref{sec:FDAs} and~\ref{sec:FDAs-egs}.  In particular, we describe them in terms of generators and relations.  We also show that they are all isomorphic over certain coefficient rings.  In Section~\ref{sec:FHAs}, we define the formal (affine) Hecke algebras and describe them in terms of generators and relations.  We prove various properties about them in Section~\ref{sec:FHAs-egs}.  In particular, we show that they are all isomorphic over certain coefficient rings, an analogue of a result of Lusztig (\cite[Thm.~9.3]{Lus89}).

\medskip

\paragraph{\textbf{Acknowledgements}} The authors would like to thank Sam Evens, Iain Gordon, Anthony Licata and Erhard Neher for useful discussions.  They would also like to thank Changlong Zhong for sharing with them some of his computations.

%
\section{Formal group laws} \label{sec:formal-group-rings}
%

In the present section, we recall the definition and properties of formal group laws (see \cite[Ch.~1,~\S3, Ch.~III,~\S1]{Fro68} and \cite[Ch.~1 and~2]{LM07} for details).

\begin{defin}[Formal group law]
  A one-dimensional commutative \emph{formal group law} (\fgl) is a pair $(R,F)$, where $R$ is a commutative ring, called the \emph{coefficient ring}, and $F=F(u,v) \in R\llbracket u,v \rrbracket$ is a power series satisfying the following axioms:
  \begin{itemize}
    \item[(FG1)] $F(u,0) = F(0,u) = u \in R\llbracket u \rrbracket$,
    \item[(FG2)] $F(u,v)= F(v,u)$, and
    \item[(FG3)] $F(u,F(v,w)) = F(F(u,v),w) \in R\llbracket u,v,w\rrbracket $.
  \end{itemize}
\end{defin}

Note that axioms~(FG1) and~(FG2) imply that
\begin{equation}\label{equ-fgl} \ts
  F(u,v) = u + v + \sum_{i,j \ge 1} a_{ij} u^i v^j,\;\text{ where }a_{ij}=a_{ji}\in R.
\end{equation}
Given an integer $m\ge 1$ we use the notation
\begin{gather*}
  u +_F v :=F(u,v),\quad
  m \cdot_F u :=\underbrace{u+_F \dotsb +_F u}_{m \text{ times}}, \quad
  \text{and }
  (-m)\cdot_F u :=-_F(m\cdot_F u),
\end{gather*}
where $-_F u$ denotes the {\em formal inverse} of $u$, i.e.\ the unique power series in $R\llbracket u\rrbracket $ such that $u +_F(-_F u) = (-_F u)+_F u=0$ (see \cite[Ch.~1, \S3, Prop.~1]{Fro68}).  We define
\begin{equation} \label{eq:mu-series}
  \mu_F(u) := \frac{-_F u}{-u} = 1-a_{11}u+a_{11}^2u^2-(a_{11}^3+a_{12}a_{11}-a_{22}+2a_{13})u^3+\dotsb
\end{equation}
(see \cite[(2.7)]{LM07}).  Note that $\mu_F(u)$ has a multiplicative inverse since its constant term is invertible.

Throughout the current paper, whenever a particular \fgl\ is denoted using a subscript (e.g.\ $F_A$, $F_M$, $F_L$, $F_U$), we will use the same subscript to denote various quantities associated
to that \fgl. Thus, we will write $-_A u$ for $-_{F_A} u$, $\mu_L$ for $\mu_{F_L}$, etc.

\begin{example}\label{FGL:muexamples}
  \begin{asparaenum}
    \item For the \emph{additive} \fgl\ $(\Z,F_A(u,v)=u+v)$ we have (see \cite[Example~1.1.4]{LM07})
      \[
        -_A u=-u \quad \text{and} \quad \mu_A(u) = 1.
      \]

    \item \label{example-item:multiplicative-FGL} For the
      \emph{multiplicative} \fgl\ $(R,F_M(u,v)=u+v-\beta uv)$,
      $\beta\in R$, $\beta\neq 0$, we have (see \cite[Example~1.1.5]{LM07})
      \begin{gather*} \ts
        -_M u=-u\sum_{i\ge 0}\beta^i u^i \quad \text{and} \quad \mu_M(u)=\sum_{i\ge 0}\beta^i u^i.
      \end{gather*}
      Observe that $(1-\beta u) \mu_M(u)=1$, so $\mu_M(u)^{-1}=1-\beta u$ in $R\llbracket u\rrbracket $.  If $\beta \in R^\times$, where $R^\times$ denotes the group of invertible elements of $R$, we say that the \fgl\ is \emph{multiplicative periodic}.

   \item The \emph{Lorentz} \fgl\ $(R,F_L)$ is given by
     \[ \ts
       F_L (u,v) = \frac{u+v}{1+\beta uv} = (u+v)\sum_{i \geq 0} (-\beta uv)^i,\quad
       \beta \in R,\ \beta \ne 0.
     \]
     We have $-_L u = -u$ and $\mu_L(u) = 1$.  Note that for $\beta = 1/c^2$, where $c$ is the speed of light, the expression $F_L(u,v)$ corresponds to the addition of relativistic parallel velocities.

  \item \label{eg-item:elliptic-FGL} Let $E$ be the elliptic curve defined by the Tate model (\cite[\S3]{Tate}):
    \begin{equation} \label{eq:elliptic-curve}
      E: \quad v = u^3 + a_1 uv + a_2 u^2 v + a_3 v^2 + a_4 u v^2 + a_6 v^3.
    \end{equation}
    Here the coefficient ring is $R=\Z[a_1,a_2,a_3,a_4,a_6]$.  The group law on $E$ induces an \emph{elliptic} \fgl\ $(R,F_E)$ with
    \[
      F_E (u, v) = u + v - a_1 u v - a_2 (u^2 v + uv^2) - 2 a_3 (u^3 v + u v^3) + (a_1 a_2 - 3 a_3) u^2 v^2 + O(5)
    \]
    (see \cite[Appendix~1, (3.6)]{Lan87}).
    We have
    \[ \ts
      -_Eu = \frac{-u}{1 - a_1 u - a_3 v(u)},\quad \mu_E(u) = \frac{1}{1-a_1u-a_3v(u)}
    \]
    (see \cite[\S IV.1,p.~120]{Sil09}), where $v(u)$ is considered as an element in $R\llbracket u\rrbracket$ after a recursive procedure in the Tate model.

  \item \label{eg-item:UFGL} We define the \emph{Lazard ring} $\LL$ to be the commutative ring with generators $a_{ij}$, $i,j \in \N_+$, and subject to the relations that are forced by the axioms for formal group laws.  The corresponding \fgl\ $(\LL,F_U(u,v)=u+v+\sum_{i,j\ge 1}a_{ij}u^iv^j)$ is then called the \emph{universal} \fgl\ (see \cite[\S1.1]{LM07}).  The series $\mu_U(u)$ is given by~\eqref{eq:mu-series}.
  \end{asparaenum}
\end{example}

Let $(R,F)$ and $(R,F')$ be formal group laws.  A \emph{morphism of formal group laws} $f\colon(R,F) \to (R,F')$ is a formal power series $f \in R\llbracket u\rrbracket $ such that $f(u+_F v)=f(u) +_{F'} f(v)$.  Given a \fgl\ $F$ over $R$, there is an isomorphism of {\fgl}s after tensoring with $\Q$,
\[
  e_F\colon (R_\Q,F_A)\to (R_\Q, F),\; R_\Q=R\otimes_\Z \Q,
\]
given by the \emph{exponential} series $e_F(u)\in R_\Q\llbracket u\rrbracket$ which satisfies the property
$e_F(u+v)=e_F(u) +_F e_F(v)$ (see \cite[Ch.~IV, \S1]{Fro68}).

\begin{example} \label{eg:exp-series}
  \begin{enumerate}
    \item \label{eg-item:exp-general} For a general \fgl\ $F(u,v)=u+v+a_{11}uv+a_{12}(u^2v+uv^2)+O(4)$ we have
        \[ \ts
          e_F(u)=u+\frac{a_{11}}{2!} u^2+\frac{a_{11}^2+2a_{12}}{3!}u^3 + O(4).
        \]

    \item For the multiplicative \fgl\ we have
        \[ \ts
          e_M(u)=\sum_{i\ge 1} (-\beta)^{i-1}\frac{u^i}{i!}, \quad \text{so that} \quad \beta e_M(u)=1-\exp(-\beta u).
        \]

    \item For the Lorentz \fgl\ we have $e_L(u) = \frac{e^{2u}-1}{e^{2u}+1}$.

    \item For the elliptic \fgl\ we have
        \[ \ts
          e_E (u) =  u - \frac{a_1}{2!} u^2 + \left( \frac{3 a_1^2 - 2 (a_1^2 + a_2)}{3!} \right) u^3 + O(4).
        \]
  \end{enumerate}
\end{example}

%
\section{Formal group algebras} \label{sec:FGAs}
%

Following \cite[\S2]{CPZ09}, we recall the definition and basic properties of formal group algebras.  These will play a fundamental role in our definition of formal (affine) Demazure and Hecke algebras.

\begin{defin}[Formal group algebra] \label{def:FGA}
  Suppose $(R,F)$ is a \fgl\ and $\Lw$ is an abelian group.  Let $R[x_\Lw] := R[\{x_\lambda\ |\ \lambda \in \Lw\}]$ denote the polynomial ring over $R$ with variables indexed by $\Lw$.  Let $\varepsilon \colon R[x_\Lw] \to R$ be the augmentation homomorphism which maps all $x_\lambda$, $\lambda \in \Lw$, to $0$ and consider the $(\ker \varepsilon)$-adic topology on $R[x_\Lw]$.  We define $R\llbracket  x_\Lw \rrbracket $ to be the $(\ker \varepsilon)$-adic completion of the polynomial ring $R [x_\Lw]$.  In particular, if $\Lw$ is finite of order $n$, then the ring $R\llbracket  x_\Lw \rrbracket $ is the usual ring of power series in $n$ variables.

  Let $J_F$ be the closure of the ideal generated by the elements $x_0$ and $x_{\lambda_1+\lambda_2} - (x_{\lambda_1}+_F x_{\lambda_2})$ for all $\lambda_1,\lambda_2 \in \Lw$.  We define the \emph{formal group algebra} (or \emph{formal group ring}) to be the quotient (see \cite[Def.~2.4]{CPZ09})
  \[
    \FA := R \llbracket  x_\Lw \rrbracket  /J_F.
  \]
  The class of $x_\lambda$ in $\FA$ will be denoted by the same letter.  By definition, $\FA$ is a complete Hausdorff $R$-algebra with respect to the $(\ker \varepsilon)$-adic topology, where $\varepsilon\colon \FA \to R$ is the induced augmentation map.  We define the augmentation ideal $\mathcal{I}_F := \ker \varepsilon$ to be the kernel of this induced map.
\end{defin}

The assignment of the formal group algebra $\FA$ to the data $(R,F,\Lambda)$ is functorial in the following ways (see \cite[Lem.~2.6]{CPZ09}).
\begin{enumerate}
  \item Given a morphism $f\colon (R,F) \to (R,F')$ of {\fgl}s, there is an induced continuous ring homomorphism $f^\star\colon R\llbracket \Lw\rrbracket _{F'} \to \FA$, $x_\lambda\mapsto f(x_\lambda)$.  If $f'\colon (R,F') \to (R,F'')$ is another morphism of {\fgl}s, then $(f'f)^\star = f^\star (f')^\star$.

  \item Given a group homomorphism $f\colon \Lw\to \Lw'$, there is an induced continuous ring homomorphism $\widehat f\colon \FA \to R\llbracket \Lw'\rrbracket _F$, $x_{\lambda}\mapsto x_{f(\lambda)}$.  If $f'\colon \Lambda' \to \Lambda''$ is another group homomorphism, then $\widehat{f'f} = \widehat{f'} \widehat f$.
\end{enumerate}
Note that maps of the type $\hat f$ commute with maps of the type $f^\star$.

\begin{example}
  The map $x_m \mapsto m \cdot_F x$, $m \in \Z$, defines $R$-algebra isomorphisms
  \[
    R\llbracket \Z\rrbracket _F \cong R\llbracket x\rrbracket  \quad \text{and} \quad R\llbracket \Z/n\Z\rrbracket _F \cong R\llbracket x \rrbracket  /(n\cdot_F x).
  \]
  More generally, there is a (noncanonical) $R$-algebra isomorphism (see \cite[Cor.~2.12]{CPZ09})
  \[
    R\llbracket \Z^n\rrbracket _F\cong R\llbracket x_1,\dotsc,x_n\rrbracket ,
  \]
  where the right hand side is independent of $F$.  This implies that if $R$ is a domain, then
  so is $R\llbracket \Z^n\rrbracket _F$.

 It follows from \eqref{equ-fgl} that $n \cdot_F x = nx + x^2 p(x)$ for some $p(x) \in R\llbracket x\rrbracket$.  Thus, if $n \in R^{\times}$, then $n \cdot_F x$ is the product of $x$ and a unit in $R\llbracket x\rrbracket $, so $(x) = (n\cdot_F x)$ and   $R\llbracket \Z/n\Z\rrbracket _F \cong R$.
\end{example}

\begin{lem}\label{lem:murel}
  Given a \fgl\ $(R,F)$, we have $\mu_F(x_\lambda)^{-1}=\mu_F(x_{-\lambda}),\;\text{ for all }\lambda\in \Lw$.
\end{lem}

\begin{proof}
  This follows immediately from the fact that $-_F x_\lambda=x_{-\lambda}$ in $\FA$.
\end{proof}

We now consider what happens at a finite (truncated) level in $\FA$.  Let $R[\Lw]_F$ denote the subalgebra of $\FA$ equal to the image of $R[x_\Lw]$ under the composition $R[x_\Lw] \hookrightarrow R \llbracket x_\Lw \rrbracket  \twoheadrightarrow \FA$. Then $\FA$ is the completion of $R[\Lw]_F$ at the ideal $(\ker \varepsilon) \cap R[\Lw]_F$. As before, the assignment $(R,F,\Lambda) \mapsto R[\Lw]_F$ is functorial with respect to group homomorphisms.

\begin{example} \label{eg:FGAs}
  \begin{asparaenum}
    \item Suppose $\Lw$ is a free abelian group.  Then for the additive \fgl\ $F_A(u,v)=u+v$ over $R$ we have ring isomorphisms (cf.~\cite[Example~2.19]{CPZ09})
      \[
        R\llbracket \Lw\rrbracket _A \cong S_R^*(\Lambda)^\wedge := \prod_{i=0}^{\infty} S^i_R(\Lw)\quad \text{and}\quad R [\Lw]_A \cong S_R^*(\Lambda) := \bigoplus_{i=0}^{\infty} S^i_R(\Lw),
      \]
      where $S^i_R(\Lw)$ is the $i$-th symmetric power of $\Lw$ over $R$, and the isomorphisms are induced by
      sending $x_\lambda$ to $\lambda\in S^1_R(\Lw)$.

    \item \label{eg-item:FGA-mult} Consider the group ring
      \[ \ts
        R[\Lw]:=\left\{\sum_j r_j e^{\lambda_j}\mid r_j\in R,\; \lambda_j\in \Lw\right\}.
      \]
      Let $\varepsilon\colon R[\Lw] \to R$ be the augmentation map, i.e.\ the $R$-linear map sending all
      $e^\lambda$, $\lambda \in \Lw$, to $1$. Let $R[\Lw]^\wedge$ be the completion of $R[\Lw]$ at $\ker \varepsilon$.

      Assume that $\beta\in R^\times$.  Then for the multiplicative periodic \fgl\ $F_M(u,v)=u+v-\beta uv$ over $R$, we have $R$-algebra isomorphisms (cf.~\cite[Example~2.20]{CPZ09})
      \[
        R\llbracket \Lw\rrbracket _M \cong R[\Lw]^\wedge \quad \text{and} \quad R[\Lw]_M \cong R[\Lw]
      \]
      induced by $x_\lambda \mapsto \beta^{-1}(1- e^{-\lambda})$ and $e^\lambda\mapsto (1-\beta x_{-\lambda})=(1-\beta x_{\lambda})^{-1}$ respectively.  Using this identification, along with Example~\ref{FGL:muexamples}\eqref{eg-item:FGA-mult} and Lemma~\ref{lem:murel}, we obtain
      \[
        \mu_M(x_\lambda)\mu_M(x_{\lambda'})=(1-\beta x_{-\lambda})(1-\beta x_{-\lambda'})=e^{\lambda+\lambda'}=1-\beta x_{-\lambda-\lambda'}=\mu_M(x_{\lambda+\lambda'}).
      \]
  \end{asparaenum}
\end{example}

\begin{example} \label{eg:FGL-torus}
  Fix a generator $\gamma$ of $\Z$ and let $t=e^\gamma$ be the corresponding element in the group ring $R[ \Z]$.  According to the previous examples, we have $R$-algebra isomorphisms
  \[
    R\llbracket \Z\rrbracket _M \cong R[t,t^{-1}]^\wedge \quad \text{and} \quad R\llbracket \Z\rrbracket _A \cong R\llbracket \gamma\rrbracket,
  \]
  where $R[t,t^{-1}]^\wedge$ denotes the completion of $R[t,t^{-1}]$ at the ideal generated by $t-1$.  At the truncated levels, we have
  \[
    R[\Z]_M \cong R[t,t^{-1}] \quad \text{and} \quad R[\Z]_A \cong R[\gamma],
  \]
  given by  $x_{n\gamma} \mapsto \beta^{-1}(1-t^{-n})$ (with inverse map given by $t \mapsto 1-\beta x_{-\gamma}$) and $x_{n\gamma} \mapsto n\gamma$ respectively.
\end{example}

%
\section{Formal Demazure operators} \label{sec:formal-dem-ops}
%

In the present section, we introduce, following \cite[\S3]{CPZ09}, the notion of formal Demazure operators.  We also state some of their properties that will be needed in our constructions.  For the remainder of the paper, we assume that $R$ is a commutative domain.

Consider a reduced root system $(\Lw,\Phi,\varrho)$ as in \cite[\S 1]{Dem73}, i.e.\ a free $\Z$-module $\Lw$ of finite rank (the \emph{weight lattice}), a finite subset $\Phi$ of $\Lw$ whose elements are called \emph{roots}, and a map $\varrho \colon \Lw \to \Lw^\vee := {\rm Hom}_\Z (\Lw,\Z)$ associating a \emph{coroot} $\alpha^\vee \in \Lw^\vee$ to every root $\alpha$, satisfying certain axioms.  The \emph{reflection} map $\lambda \mapsto \lambda - \langle \alpha^\vee, \lambda \rangle \alpha$ is denoted by $s_\alpha$.  Here $\langle \cdot, \cdot \rangle$ denotes the natural pairing between $\Lw^\vee$ and $\Lw$.

The \emph{Weyl group} $W$ associated to a reduced root system is the subgroup of linear automorphisms of $\Lw$ generated by the reflections $s_\alpha$.  We fix sets of simple roots $\{\alpha_i\}_{i \in I}$ and fundamental
weights $\{\omega_i\}_{i \in I}$.  That is, $\omega_i \in \Lambda^\vee$ satisfies $\langle \omega_i, \alpha_j \rangle = \delta_{ij}$ for all $i,j \in I$.  Let $\{s_i = s_{\alpha_i}\}_{i \in I}$ denote the corresponding set of simple reflections in $W$ and let $\ell$ denote the usual length function on $W$.  We say the root system is \emph{simply laced} if $\left< \alpha_i^\vee, \alpha_j \right> \in \{0,-1\}$ for all $i,j \in I$, $i \ne j$.  For instance, the roots systems of type ADE are simply laced.  For $i,j \in I$, $i \ne j$, set $m_{ij}$ equal to 2, 3, 4 or 6 if the product $\langle \alpha_i^\vee, \alpha_j \rangle \langle \alpha_j^\vee, \alpha_i \rangle$ is equal to 0, 1, 2 or 3 respectively and set $m_{ij}=0$ if $\langle \alpha_i^\vee, \alpha_j \rangle \langle \alpha_j^\vee, \alpha_i \rangle \ge 4$.  Then $m_{ij}$ is the order of $s_i s_j$ if $m_{ij} > 0$, and $m_{ij}=0$ if and only if $s_i s_j$ has infinite order (see, for example, \cite[Prop.~3.13]{Kac90}).

Fix a \fgl\ $(R,F)$.  Since the Weyl group acts linearly on $\Lw$, it acts by $R$-algebra automorphisms on $\FA$ via the functoriality in $\Lw$ of $\FA$ (see Section~\ref{sec:FGAs}), i.e.\ we have
\[
  w(x_\lambda) = x_{w(\lambda)},\text{ for all } w \in W,\ \lambda \in \Lw.
\]

\begin{defin}[Formal Demazure operator $\Delta^F_\alpha$] \label{def:Dem-operator}
  By \cite[Cor.~3.4]{CPZ09}, for any $\varphi \in \FA$ and root $\alpha \in \Phi$,  the element $\varphi - s_\alpha (\varphi)$ is uniquely divisible by $x_{\alpha}$. We define an $R$-linear operator $\Delta_\alpha^F$ on $\FA$ (see \cite[Def.~3.5]{CPZ09}), called the \emph{formal Demazure operator}, by
  \[ \ts
    \Delta_\alpha^F( \varphi ):=\frac{\varphi - s_\alpha(\varphi)}{x_\alpha},\quad \varphi \in \FA.
  \]
  Observe that if $F$ is the additive or multiplicative \fgl,
  then $\Delta_\alpha^F$ is the classical Demazure operator of \cite[\S3 and~\S9]{Dem73}.  We will often omit the superscript $F$ when the \fgl\ is understood.
\end{defin}

\begin{defin}[$g^F$, $\kappa^F_\alpha$ and $C^F_\alpha$] \label{Def:Coper}
  Consider the power series $g^F(u,v)$ defined by $u+_F v = u+v-uvg^F(u,v)$ and, for $\alpha \in \Phi$, let
  \[ \ts
    \kappa^F_\alpha := g^F(x_{\alpha},x_{-\alpha}) = \frac{1}{x_\alpha} + \frac{1}{x_{-\alpha}} \in \FA.
  \]
  We define an $R$-linear operator $C_\alpha^F$ on $\FA$ (see \cite[Def.~3.11]{CPZ09}) by
  \[
    C_\alpha ^F(\varphi) := \kappa_\alpha^F \varphi -\Delta_{\alpha}^F(\varphi),\quad \varphi \in \FA.
  \]
  We will often omit the superscript $F$ when the \fgl\ is understood.
\end{defin}

\begin{lem} \label{lem:kappa-eq-zeo}
  The following statements are equivalent.
  \begin{enumerate}
    \item \label{lem-item:Fdivis} $F(u,v) = (u+v)h(u,v)$ for some $h(u,v) \in R \llbracket u,v \rrbracket$.
    \item \label{lem-item:kappa-zero} $\kappa^F_\alpha = 0$ for all $\alpha \in \Phi$.
    \item \label{lem-item:kappa_alpha-zero} $\kappa^F_\alpha = 0$ for some $\alpha \in \Phi$.
    \item \label{lem-item:mu-one} $\mu_F(u)=1$.
  \end{enumerate}
  If these equivalent conditions are satisfied, we write $\kappa^F=0$.  If they are not satisfied, we write $\kappa^F \ne 0$.
\end{lem}

\begin{proof}
  First suppose that $F(u,v) = (u+v)h(u,v)$ for some $h(u,v) \in R \llbracket u,v \rrbracket$.  Then, for any $\alpha \in \Phi$, we have $F(x_\alpha,-x_\alpha)=0$.  By the uniqueness of the formal inverse, this implies that $x_{-\alpha} = -_F x_\alpha = -x_\alpha$.  Thus $\kappa^F_\alpha = 0$ and so~\eqref{lem-item:Fdivis} implies~\eqref{lem-item:kappa-zero}.  Clearly~\eqref{lem-item:kappa-zero} implies~\eqref{lem-item:kappa_alpha-zero}.

  Now suppose that $\kappa^F_\alpha = 0$ for some $\alpha \in \Phi$.  Then $x_\alpha \in \mathcal{I}_F \setminus \{0\}$ and $F(x_\alpha,-x_\alpha)=0$ in $R \llbracket \Lambda \rrbracket_F$.  By the definition of a \fgl, we have
  \begin{gather*} \ts
    F(x_\alpha,-x_\alpha)=x_\alpha^2\sum_{i,j\ge 1}(-1)^j a_{ij}x_\alpha^{i+j-2}=x_\alpha^2\sum_{n\ge 0}b_n x_\alpha^n \in \mathcal{I}_F^2, \\ \ts \text{where }b_n= \sum_{i+j=n+2} (-1)^ja_{ij}.
  \end{gather*}
  We claim that $b_n=0$ for all $n\ge 0$.  Indeed, let $n_0$ be the smallest $n$ such that $b_{n_0}\neq 0$. Then
  \[ \ts
    0 = F(x_\alpha,-x_\alpha)=x_\alpha^{2}\sum_{n\ge n_0}b_n x_\alpha^n=x_\alpha^{2+n_0}\sum_{n\ge n_0}b_n x_\alpha^{n-n_0}.
  \]
  Since $x_\alpha \ne 0$ (this follows from~\cite[Lem.~4.2]{CPZ09}) and $R \llbracket \Lambda \rrbracket_F$ is a domain, we have
  \[ \ts
    \sum_{n \ge 0} b_n x_\alpha^{n-n_0} = 0.
  \]
  Applying the augmentation map, we obtain $b_{n_0}=0$, contradicting our choice of $n_0$.

  Now let $F_i(u,v)$ be the $i$-th homogeneous component of $F$, $i \ge 2$.  Since $F_i(u,-u)=b_{i-2}u^i=0$, $F_i(u,v)$ is divisible by $(u+v)$ for every $i\ge 2$.  Since the homogeneous components of degree zero and one for any \fgl\ are 0 and $u+v$, this implies that $F(u,v) = (u+v)h(u,v)$ for some $h(u,v) \in R \llbracket u,v \rrbracket$.  Thus~\eqref{lem-item:kappa_alpha-zero} implies~\eqref{lem-item:Fdivis}.

  Now, if $\mu_F(u)=1$, then $- x_\alpha = -_F x_\alpha = x_{-\alpha}$ and so $\kappa^F_\alpha = 0$ for all $\alpha \in \Phi$.  Thus~\eqref{lem-item:mu-one} implies~\eqref{lem-item:kappa-zero}.

  Finally, if $F(u,v) = (u+v)h(u,v)$ for some $h(u,v) \in R \llbracket u,v \rrbracket$, then $-_F u = -u$ by the uniqueness of the formal inverse (as above).  Thus $\mu_F(u)=1$.  Hence~\eqref{lem-item:Fdivis} implies~\eqref{lem-item:mu-one}.
\end{proof}

As in the case of the usual Demazure operators, the operators $\Delta_\alpha^F$ and $C_\alpha^F$ satisfy Leibniz-type properties (see \cite[Props.~3.8 and~3.12]{CPZ09}).

%
\section{Algebraic oriented cohomology theories and characteristic maps} \label{sec:AOCT}
%

While the new objects to be defined in the current paper rely, for the most part, only on a formal group law, the motivation behind these definitions is geometric.  For this reason, we now recall several facts concerning algebraic oriented cohomology theories.  We refer the reader to \cite{LM07} and \cite{Pan03} for further details and examples on algebraic oriented cohomology theories.

An \emph{algebraic oriented cohomology theory} (AOCT) is a contravariant functor $\hh$ from the category of smooth projective varieties over a field $k$ to the category of commutative unital rings which satisfies certain properties (see \cite[\S1.1]{LM07}).  Given a morphism $f\colon X\to Y$ of varieties, the map $\hh(f)$ will be denoted $f^*$ and called the \emph{pullback} of $f$.  One of the characterizing properties of $\hh$ is that, for any proper map $f \colon X\to Y$, there is an induced map $f_*\colon \hh(X) \to \hh(Y)$ of $\hh(Y)$-modules called the \emph{push-forward} (here $\hh(X)$ is an $\hh(Y)$-module via $f^*$).  A \emph{morphism} of AOCTs is a natural transformation of functors that also commutes with push-forwards.  Basic examples of AOCTs are Chow groups $\CH$ and Grothendieck's $K_0$ (see \cite[\S\S2.1, 2.5, 3.8]{Pan03} for further examples).

The connection between algebraic oriented cohomology theories and {\fgl}s is as follows.
Given two line bundles $L_1$ and $L_2$ over $X$, we have (see \cite[Lem.~1.1.3]{LM07})
\[
  \ch_1 (L_1 \otimes L_2) = \ch_1 (L_1) +_F \ch_1 (L_2),
\]
where $\ch_1$ is the first characteristic class with values in $\hh$ and $F$ is a one-dimensional commutative \fgl\  over the coefficient ring $R=\hh(\Spec k)$ associated to $\hh$.

There is an AOCT $\Omega$ defined over a field of characteristic zero, called \emph{algebraic cobordism} (see \cite[\S1.2]{LM07}), that is universal in the following sense: Given any AOCT $\hh$ there is a unique
morphism $\Omega \to \hh$ of AOCTs. The \fgl\ associated to $\Omega$ is the universal \fgl\ $F_U$.

Moreover, given a \fgl\  $F$ over a ring $R$ together with a morphism $\LL \to R$, we define a functor $X\mapsto
\hh(X):=\Omega(X)\otimes_\LL R$.  Over a field of characteristic zero, the functor $\hh$ gives an AOCT.

\begin{example}
  In the above terms, the additive \fgl\ corresponds to the theory of Chow groups.  The multiplicative periodic \fgl\ with $\beta \in R^\times$ corresponds to Grothendieck's $K_0$. The multiplicative \fgl\ with $\beta\notin R^\times$ corresponds to connective $K$-theory.
\end{example}

Let $G$ be a split simple simply connected linear algebraic group over a field $k$ corresponding to the root system $(\Lw,\Phi,\varrho)$.  Fix a split maximal torus $T$ and a Borel subgroup $B$ so that $T \subseteq B \subseteq G$. Let $G/B$ be the variety of Borel subgroups of $G$ and let $F$ be the \fgl\ over $R$ associated to an AOCT  $\hh$ satisfying the assumptions of \cite[Thm.~13.12]{CPZ09}.  Consider the formal group algebra $\FA$. Then there is a ring homomorphism, called the \emph{characteristic map} (see \cite[\S6]{CPZ09}),
\[
  \cc_F \colon \FA \to \hh(G/B),\quad x_\lambda \mapsto \ch_1(L(\lambda)),
\]
where $L(\lambda)$ is the line bundle associated to $\lambda \in \Lambda$.  Note that this map is neither injective nor surjective in general.  Its kernel contains the ideal generated by $W$-invariant elements, and $\hh(G/B)$ modulo the ideal generated by the image of $\cc_F$ is isomorphic to $\hh(G)$ (see \cite[Prop.~5.1]{GZ12}).

\begin{example}
  \begin{asparaenum}
    \item The characteristic map for the theory of Chow groups, i.e.\ corresponding to the additive \fgl, is given by
        \[
          \mathfrak{c}_A \colon \Z \llbracket\Lw\rrbracket_A \to \CH(G/B),\quad x_\lambda \mapsto c_1 \left( L(\lambda) \right),
        \]
        which recovers the usual characteristic map  for Chow groups (see \cite[\S1.5]{dem-des}).

    \item The characteristic map for Grothendieck's $K_0$, i.e.\ corresponding to the multiplicative periodic \fgl, is given by
        \[
          \cc_M \colon\Z \llbracket\Lw\rrbracket_M \to K_0(G/B),\quad x_\lambda \mapsto 1 -[L(\lambda)^\vee].
        \]
        Restricting to the integral group ring $\Z[\Lw]$ and using the identification of Example~\ref{eg:FGAs}\eqref{eg-item:FGA-mult}, we recover the usual characteristic map for $K_0$
        (\cite[\S1.6]{dem-des}) which maps $e^\lambda$ to $[L(\lambda)]$.

    \item Algebraic cobordism $\Omega$ defined over a field of characteristic 0 satisfies the assumptions of \cite[Thm.13.12]{CPZ09}.  Therefore, we have the characteristic map
        \[
          \cc_U\colon \LL\llbracket\Lw\rrbracket_U \to \Omega(G/B),\quad x_\lambda \mapsto \mathrm{c}^\Omega_1(L(\lambda)).
        \]
    \end{asparaenum}
\end{example}

Let $G/P_i$ be the projective homogeneous variety, where $P_i$ is the minimal parabolic subgroup of $G$ corresponding to the simple root $\alpha_i$, $i \in I$. Then
\[
  p\colon G/B=\mathbb{P}_{G/P_i}(1\oplus L(\omega_i)) \to G/P_i
\]
is the projective bundle associated to the vector bundle $1 \oplus L(\omega_i)$, there $1$ denotes the trivial bundle of rank one (see, for example, \cite[\S10.3]{CPZ09}).  Then the operators $C_\alpha^F$ introduced in Definition~\ref{Def:Coper} have the following geometric interpretation in terms of \emph{push-pull operators} (generalizing \cite[Prop.]{PiRa99}).

\begin{prop}[{\cite[Prop.~10.10(4)]{CPZ09}}] \label{prop:push-pull-ops} We have
\[
  p^* p_* (\cc_F(\chi)) = \cc_F(C_\alpha^F(\chi)),\;\text{ for all }\chi\in \FA.
\]
\end{prop}

%
\section{Formal (affine) Demazure algebras: definitions} \label{sec:FDAs}
%

In the present section, we introduce the notion of a twisted formal group algebra and a particular subalgebra, called the formal (affine) Demazure algebra, which is one of our main objects of interest.  Our method is inspired by the approach of \cite[\S4.1]{KoKu86}.

\begin{defin}[Twisted formal group algebra] \label{def:twisted-FGA}
  Let $Q^F = Q^{(R,F)}$ denote the subring of the field of fractions of $\FA$ generated by $\FA$ and $\{x_\lambda^{-1}\ |\  \lambda \in \Lambda \setminus \{0\}\}$.  The action of the Weyl group $W$ on $\FA$ induces an action by automorphisms on $Q^F$. We define the \emph{twisted formal group algebra} to be the smash product $Q_W^F := R[W] \ltimes_R Q^F$.  (This is sometimes denoted by $R[W] \# Q^F$.)  In other words, $Q_W^F$ is equal to $R[W] \otimes_R Q^F$ as an $R$-module, with multiplication given by
  \[
    (\delta_{w'} \psi') (\delta_w \psi) = \delta_{w'w} w^{-1}(\psi')\psi\; \text{ for all } w,w'\in W,\; \psi,\psi' \in Q^F
  \]
  (extended by linearity), where $\delta_w$ denotes the element in $R[W]$ corresponding to $w$ (so we have $\delta_{w'} \delta_w = \delta_{w'w}$ for $w,w' \in W$).
\end{defin}

Observe that $Q_W^F$ is a free right $Q^F$-module (via right multiplication) with basis $\{\delta_w\}_{w\in W}$.  Note that $Q_W^F$ is not a $Q^F$-algebra (but only an $R$-algebra) since $\delta_e Q^F=Q^F\delta_e$ is not central in $Q_W^F$. We denote $\delta_e$ (the unit element of $Q_W^F$) by $1$.

\begin{defin}[Formal Demazure element] \label{def:Demazure-element}
  For each root $\alpha \in \Phi$, we define the corresponding \emph{formal Demazure element}
  \[
    \iDelta^F_\alpha:=\tfrac{1}{x_{\alpha}} (1 - \delta_{s_\alpha}) = \tfrac{1}{x_\alpha} - \delta_{s_\alpha} \tfrac{1}{x_{-\alpha}} \in Q_W^F
  \]
  (cf.\ \cite[($\text{I}_{24}$)]{KoKu86}).  We will omit the superscript $F$ when the \fgl\ is clear from the context.
\end{defin}

We can now define our first main objects of study.

\begin{defin}[Formal (affine) Demazure algebra]
  The \emph{formal Demazure algebra} $D_F$ is the $R$-subalgebra of $Q_W^F$ generated by the formal Demazure elements $\iDelta_i^F$.  The \emph{formal affine Demazure algebra} $\bD_F$ is the $R$-subalgebra of $Q_W^F$ generated by $D_F$ and $\FA$. When we wish to specify the coefficient ring, we write $D_{R,F}$ (resp.\ $\bD_{R,F}$) for $D_F$ (resp.\ $\bD_F$).
\end{defin}

\begin{rem}
  Suppose $\hh$ is an algebraic oriented cohomology theory satisfying the assumptions of~\cite[Thm.~13.12]{CPZ09} and with \fgl\ $F$ (see Section~\ref{sec:AOCT}).  By Proposition~\ref{prop:push-pull-ops}, we see that, under the characteristic map $\cc_F$, the affine Demazure algebra $\bD_F$ corresponds to the algebra of operators on $\hh(G/B)$ generated by left multiplication (by elements of $\hh(G/B)$) and the push-pull operators $p^*p_*$.
\end{rem}

\begin{lem}[cf.\ {\cite[Prop.~4.2]{KoKu86}}] \label{lem:Dcomm}
  For all $\psi \in Q^F$ and $\alpha \in \Phi$, we have
  \[
    \psi \iDelta_\alpha=\iDelta_\alpha s_\alpha(\psi)+\Delta_{\alpha}(\psi),
  \]
  where $\Delta_{\alpha}(\psi)=\frac{\psi-s_\alpha(\psi)}{x_{\alpha}}\in Q^F$ is the formal Demazure operator applied to $\psi$ (see Definition~\ref{def:Dem-operator}).
\end{lem}

\begin{proof}
  We have
  \[
    \psi \iDelta_\alpha =  \psi \left( \tfrac{1}{x_\alpha}-\delta_{s_\alpha}\tfrac{1}{x_{-\alpha}} \right)
    = \tfrac{\psi-s_\alpha(\psi)}{x_{\alpha}} + \left( \tfrac{1}{x_\alpha} - \delta_{s_\alpha}\tfrac{1}{x_{-\alpha}} \right) s_\alpha(\psi) = \Delta_{\alpha}(\psi) + \iDelta_\alpha s_\alpha(\psi). \qedhere
  \]
\end{proof}

Proceeding from Lemma~\ref{lem:Dcomm} by induction, we obtain the following general formula (cf.\ \cite[($\text{I}_{26}$)]{KoKu86})
\[
  \psi\iDelta_{\beta_1}\iDelta_{\beta_2} \dotsm \iDelta_{\beta_s} = \sum_{(1\le i_1<\dotsb <i_r\le s)} \iDelta_{\beta_{i_1}}\iDelta_{\beta_{i_2}}\dotsm \iDelta_{\beta_{i_r}} \phi(i_1,\dotsc,i_r),
\]
where $\beta_1,\beta_2,\dotsc,\beta_s \in \Phi$ are roots and $\phi(i_1,\dotsc,i_r)\in Q^F$ is defined to be the composition $\Delta_{\beta_s}\circ \Delta_{\beta_{s-1}}\circ \dotsb \circ \Delta_{\beta_1}$ applied to $\psi$, where the Demazure operators at the places $i_1,\dotsc, i_r$ are replaced by the respective reflections.

Recall the elements $\kappa_\alpha$, $\alpha \in \Phi$, from Definition~\ref{Def:Coper}.  It is easy to verify that $\kappa_\alpha \delta_{s_\alpha}=\delta_{s_\alpha}\kappa_\alpha$, $\kappa_\alpha \iDelta_\alpha = \iDelta_\alpha \kappa_\alpha$, and
\begin{equation} \label{eq:Dem-elem-squared}
  \iDelta_\alpha^2=\iDelta_\alpha \kappa_\alpha.
\end{equation}

\begin{example}
  \begin{enumerate}
    \item For the additive and Lorentz {\fgl}s we obtain the \emph{nilpotence relation} $\iDelta_\alpha^2=0$ since $\kappa_\alpha^A = \kappa_\alpha^L= 0$.

    \item For the multiplicative \fgl\ we obtain the relation $\iDelta_\alpha^2 = \beta\iDelta_\alpha$, since $\kappa_\alpha^M = \beta$.  In particular, if $\beta=1$ we obtain the      \emph{idempotence relation} $\iDelta_\alpha^2=\iDelta_\alpha$.

    \item For the elliptic \fgl\ we have, in the notation of~Example~\ref{FGL:muexamples}\eqref{eg-item:elliptic-FGL},
        \[ \ts
          \iDelta_{\alpha}^2 = \frac{a_1 x_\alpha + a_3 v(x_\alpha)}{x_\alpha} \iDelta_\alpha.
        \]
        For example, if $a_3=0$, then $\iDelta_\alpha^2 = a_1 \iDelta_\alpha$.
\end{enumerate}
\end{example}

To simplify notation in what follows, for $i,j,i_1,\dotsc,i_k \in I$, we set
\begin{equation} \label{eq:subscript-shorthand}
  x_{\pm i}=x_{\pm \alpha_i},\ x_{\pm i \pm j}=x_{\pm \alpha_i \pm \alpha_j},\ \delta_{i_1i_2\ldots
  i_k}=\delta_{s_{i_1}s_{i_2}\dotsm s_{i_k}},\ \iDelta_{i_1 \dotsm i_k} = \iDelta_{\alpha_{i_1}} \dotsm \iDelta_{\alpha_{i_k}},\ \kappa_i=\kappa_{\alpha_i}.
\end{equation}
Furthermore, when we write an expression such as $\frac{\delta_w}{\varphi}$ for $w \in W$, $\varphi \in \FA$, we interpret this as being equal to $\delta_w \frac{1}{\varphi}$.  That is, we consider the numerators of rational expressions to be to the left of their denominators.

\begin{lem}
  For all $\lambda, \nu \in \Lambda \setminus \{0\}$ with $\lambda + \nu \ne 0$, the element
  \begin{equation} \label{eq:kappa-lambda-mu}
    \kappa_{\lambda, \nu} := \tfrac{1}{x_{\lambda + \nu} x_\nu} - \tfrac{1}{x_{\lambda+\nu} x_{-\lambda}} - \tfrac{1}{x_\lambda x_\nu}
  \end{equation}
  lies in $\FA$.
\end{lem}

\begin{proof}
  We have
  \[
    g(x_{\lambda+\nu},x_{-\lambda}) = \tfrac{x_{\lambda+\nu} + x_{-\lambda}-x_\nu}{x_{\lambda+\nu}x_{-\lambda}} = \tfrac{1}{x_{-\lambda}} + \tfrac{1}{x_{\lambda+\nu}} - \tfrac{x_\nu}{x_{\lambda+\nu}x_{-\lambda}} \in \FA.
  \]
  and, hence,
  \[ \ts
    \kappa_{\lambda, \nu}=\frac{g(x_{\lambda+\nu},x_{-\lambda})-g(x_\lambda,x_{-\lambda})}{x_\nu},
  \]
  where $g(x_\lambda,x_{-\lambda})=\tfrac{1}{x_\lambda}+\tfrac{1}{x_{-\lambda}}\in \FA$.  Therefore, it suffices to show that $g(x_{\lambda+\nu},x_{-\lambda})-g(x_\lambda,x_{-\lambda})$ is divisible by $x_\nu$.  The latter follows (taking $u_1=x_\lambda$, $u_2=x_\nu$ and $v=x_{-\lambda}$) from the congruence $u_1 +_F u_2 \equiv u_1$ (mod $u_2$), which implies that $g(u_1 +_F u_2, v) \equiv g(u_1,v)$ (mod $u_2$).
\end{proof}

In what follows, for $i,j \in I$ and $n_1,n_2,m_1,m_2 \in \Z$, we will write $\kappa_{n_1 i + m_1j, n_2 i + m_2 j}$ for $\kappa_{n_1 \alpha_i + m_1 \alpha_j, n_2 \alpha_i + m_2 \alpha_j}$.  For example,
\begin{equation} \label{eq:kappa_ij}
  \kappa_{i,j}=\tfrac{1}{x_{i+j}} \Big( \tfrac{1}{x_{j}} - \tfrac{1}{x_{-i}} \Big) - \tfrac{1}{x_{i}x_{j}} \in \FA.
\end{equation}

\begin{prop} \label{prop:FDA-braid-relations}
  Suppose $i,j \in I$ and let $m_{ij}$ be the order of $s_i s_j$ in $W$.  Then
  \begin{equation} \label{eq:Demazure-general-braid-relation}
    \underbrace{\iDelta_j \iDelta_i \iDelta_j \dotsm}_{m_{ij} \text{ terms }} - \underbrace{\iDelta_i \iDelta_j \iDelta_i \dotsm}_{m_{ij} \text{ terms}} = \sum_{w \in W,\, 1 \le \ell(w) \le m_{ij}-2} \iDelta_w \eta_w^{ji}
  \end{equation}
  for some $\eta_w^{ji} \in Q^F$.  In particular, we have the following:
  \begin{enumerate}
    \item \label{prop-item:Delta_i-nonadjacent} If $\langle \alpha_i^\vee, \alpha_j \rangle=0$, so that $m_{ij}=2$, then $\iDelta_i \iDelta_j = \iDelta_j \iDelta_i$.

    \item \label{prop-item:Delta_i-adjacent} If $\langle \alpha_i^\vee, \alpha_j \rangle = \langle \alpha_j^\vee, \alpha_i \rangle=-1$ so that $m_{ij}=3$, then
        \begin{equation} \label{eq:Dem-A-braid-relation}
          \iDelta_j \iDelta_i \iDelta_j - \iDelta_i \iDelta_j \iDelta_i = \iDelta_i \kappa_{i,j} - \iDelta_j \kappa_{j,i}.
        \end{equation}

    \item \label{prop-item:Delta-BC_i-adjacent} If $\langle \alpha_i^\vee, \alpha_j \rangle=-1$ and $\langle \alpha_j^\vee, \alpha_i \rangle=-2$ so that $m_{ij}=4$, then
        \begin{multline} \label{eq:Dem-BC-braid-relation}
          \iDelta_{jiji} - \iDelta_{ijij} = \iDelta_{ij} (\kappa_{i+2j,-j} + \kappa_{j,i}) - \iDelta_{ji} (\kappa_{i+j,j} + \kappa_{i,j}) \\
          + \iDelta_j \big(\Delta_i(\kappa_{i+j,j} + \kappa_{i,j})\big) - \iDelta_i \big(\Delta_j(\kappa_{i+2j,-j} + \kappa_{j,i})\big).
        \end{multline}

    \item \label{prop-item:Delta-G2-adjacent} If $\langle \alpha_i^\vee, \alpha_j \rangle=-1$ and $\langle \alpha_j^\vee, \alpha_i \rangle=-3$ so that $m_{ij}=6$, then
        \begin{multline} \label{eq:Dem-G2-braid-relation}
          \iDelta_{jijiji} - \iDelta_{ijijij} \\
          = \iDelta_{ijij} (\kappa_{j,i} + \kappa_{2i+3j,-i-2j} + \kappa_{-i-3j,i+2j} + \kappa_{i+2j,-j}) - \iDelta_{jiji} (\kappa_{i,j} + \kappa_{-2i-3j,i+2j} + \kappa_{-i-2j,i+3j} + \kappa_{i+j,j}) \\
          + \iDelta_{jij} \big( \Delta_i (\kappa_{i,j} + \kappa_{-2i-3j,i+2j} + \kappa_{-i-2j,i+3j} + \kappa_{i+j,j}) \big)
          - \iDelta_{iji} \big( \Delta_j (\kappa_{j,i} + \kappa_{2i+3j,-i-2j} + \kappa_{-i-3j,i+2j} + \kappa_{i+2j,-j}) \big) \\
          + \iDelta_{ij} \xi_{ij} - \iDelta_{ji} \xi_{ji}
          + \iDelta_j (\Delta_i(\xi_{ji})) - \iDelta_i (\Delta_j(\xi_{ij}))
        \end{multline}
        where
        \begin{multline*}
          \xi_{ij} = \tfrac{1}{x_i x_{i+j} x_{i+2j} x_{2i+3j}} + \tfrac{1}{x_ix_jx_{i+2j}x_{-2i-3j}} + \tfrac{1}{x_ix_jx_{2i+3j}x_{-i-j}} - \tfrac{1}{x_ix_{i+j}x_{i+2j}x_{-i-3j}} \\
          - \tfrac{1}{x_ix_{i+j}x_{i+3j}x_{-j}} + \tfrac{1}{x_{i+j}x_{i+3j}x_{-j}x_{-2i-3j}} + \tfrac{1}{x_{i+3j}x_{2i+3j}x_{-j}x_{-i-2j}} + \tfrac{1}{{x_{i+j}x_{i+2j}x_{-i-3j}x_{-2i-3j}}} - \tfrac{1}{x_ix_jx_{i+2j}x_{i+3j}}
        \end{multline*}
        and
        \begin{multline*}
          \xi_{ji} = \tfrac{1}{x_ix_jx_{2i+3j}x_{-i-2j}} + \tfrac{1}{x_ix_jx_{i+2j}x_{-i-3j}} + \tfrac{1}{x_jx_{i+2j}x_{i+3j}x_{2i+3j}} - \tfrac{1}{x_ix_jx_{i+j}x_{2i+3j}} \\
          + \tfrac{1}{x_{i+j}x_{i+2j}x_{-i}x_{-2i-3j}} + \tfrac{1}{x_{i+3j}x_{2i+3j}x_{-i-j}x_{-i-2j}} + \tfrac{1}{x_{i+j}x_{i+3j}x_{-i}x_{-i-2j}} - \tfrac{1}{x_j x_{i+3j}x_{2i+3j}x_{-i-j}} - \tfrac{1}{x_jx_{i+j}x_{i+3j}x_{-i}}.
        \end{multline*}
  \end{enumerate}
\end{prop}

\begin{proof}
  For $\alpha \in \Phi$, let $\chi_\alpha = \frac{1}{x_\alpha}$ and $\chi_\alpha' =- \frac{1}{x_\alpha}$.  As usual, we set $\chi_{\pm i} = \chi_{\pm \alpha_i}$, $\chi_{\pm i \pm j} = \chi_{\pm \alpha_i \pm \alpha_j}$ and similarly for the primed versions.  Let
  \[
    (n,k) =
    \begin{cases}
      (j,i) & \text{if } m_{ij} \text{ is even}, \\
      (i,j) & \text{if } m_{ij} \text{ is odd}.
    \end{cases}
  \]
  Then we have
  \[ \ts
    \underbrace{\iDelta_j \iDelta_i \dotsm \iDelta_k}_{m_{ij} \text{ terms}} = (\chi_j + \chi_j' \delta_j) (\chi_i + \chi_i' \delta_i) \dotsm (\chi_k + \chi_k' \delta_k).
  \]
  Since the $\delta_w$, $w \in W$, form a basis of $Q_W^F$ as a right $Q^F$-module, this can be written as a sum of (right) $Q^F$-multiplies of $\delta_w$.  The leading term (with respect to the length of $w$) is
  \[
    \delta_{\underbrace{s_j s_i s_j \dotsm s_n s_k}_{m_{ij} \text{ terms}}} (\chi'_{-\underbrace{s_k s_n \dotsm s_j s_i}_{m_{ij}-1 \text{ terms}} (\alpha_j)} \dotsm \chi'_{-s_k(\alpha_n)} \chi_{-\alpha_k})^{-1},
  \]
  Now, by \cite[Cor.~2 de la Prop.~17]{Bou81},
  \[
    \alpha_k,\ s_k(\alpha_n), \dotsc,\ \underbrace{s_k s_n \dotsm s_j s_i}_{m_{ij-1} \text{ terms}} (\alpha_j)
  \]
  are precisely the positive roots mapped to negative roots by $\delta_{s_j s_i s_j \dotsm}$ ($m_{ij}$ reflections in the subscript).  Since $\delta_{s_j s_i s_j \dotsm} = \delta_{s_i s_j s_i \dotsm}$ ($m_{ij}$ reflections in each subscript), we see that the highest order terms in $\iDelta_j \iDelta_i \iDelta_j \dotsm - \iDelta_i \iDelta_j \iDelta_i \dotsm$ cancel.

  Now we consider the terms of order $m_{ij}-1$.  Let
  \[
    B = \underbrace{\chi_i' \delta_i \chi_j' \delta_j \dotsm \chi_k' \delta_k}_{m_{ij}-1 \text{ pairs}},\quad B' = \underbrace{\chi_j' \delta_j' \chi_i' \delta_i \dotsm \chi_n' \delta_n}_{m_{ij}-1 \text{ pairs}}.
  \]
  Then the terms of order $m_{ij}-1$ in $\iDelta_j \iDelta_i \iDelta_j \dotsm$ and $\iDelta_i \iDelta_j \iDelta_i \dotsm$ ($m_{ij}$ terms in each product) are $\chi_jB + B'\chi_k$ and $\chi_iB' + B\chi_n$ respectively.  It is well known that, if $w_0$ is the longest element of the Weyl group generated by the simple reflections $s_i$ and $s_j$, then $w_0 (\alpha_p) = -\alpha_{p'}$, $p \in \{i,j\}$, where $p'=p$ for $m_{ij}$ equal to 4 or 6 and $i'=j$, $j'=i$ for $m_{ij}=3$.  It follows that $\chi_j B = B \chi_n$ and $\chi_i B' = B' \chi_k$.  Thus the terms of order $m_{ij}-1$ in $\iDelta_j \iDelta_i \iDelta_j \dotsm - \iDelta_i \iDelta_j \iDelta_i \dotsm$ cancel.

  Now, it is easy to see from the definition of the formal Demazure elements (Definition~\ref{def:Demazure-element}) that, for $i_1,\dotsc,i_p \in I$, the element $\Delta_{i_1 \dotsm i_p}$ is equal to a nonzero multiple of $\delta_{i_1 \dotsm i_p}$ plus lower order terms.  Combined with the above, this proves~\eqref{eq:Demazure-general-braid-relation}, but with the sum on the right hand side over $w \in W$ with $\ell(w) \le m_{ij}-2$.  To complete the proof of~\eqref{eq:Demazure-general-braid-relation}, it thus remains to consider the constant terms (i.e.\ terms of order 0).  There is a natural action of $Q_W^F$ on $Q^F$, where $Q^F \subseteq Q_W^F$ acts by left multiplication and $R[W]$ acts via the action of the Weyl group.  Under this action $\iDelta_p(1) = \Delta_p(1) = 0$ for $p \in \{i,j\}$.  Thus the constant term of the right hand side of~\eqref{eq:Demazure-general-braid-relation} must be zero.

  Under the assumptions of~\eqref{prop-item:Delta_i-nonadjacent}, we have $s_i(\alpha_j)=\alpha_j$ and
  \[
    \iDelta_i\iDelta_j = \left( \tfrac{1}{x_i}-\tfrac{\delta_i}{x_{-i}} \right) \left( \tfrac{1}{x_j} - \tfrac{\delta_j}{x_{-j}} \right) = \tfrac{1}{x_ix_j} - \tfrac{\delta_i}{x_{-i}x_j} - \tfrac{1}{x_i} \tfrac{\delta_j}{x_{-j}} + \tfrac{\delta_i}{x_{-i}} \tfrac{\delta_j}{x_{-j}} = \tfrac{1}{x_ix_j} - \tfrac{\delta_i}{x_{-i}x_j} - \tfrac{\delta_j}{x_ix_{-j}} + \tfrac{\delta_{ij}}{x_{-i}x_{-j}}.
  \]
  Since the final expression is symmetric in $i$ and $j$, we have $\iDelta_i\iDelta_j=\iDelta_j\iDelta_i$.

  We now prove~\eqref{prop-item:Delta_i-adjacent}.  We have $s_i(\alpha_j) = s_j(\alpha_i) = \alpha_i+\alpha_j$ and $s_is_j(\alpha_i)=\alpha_j$.  Thus
  \begin{align*}
    \iDelta_j\iDelta_i\iDelta_j &= \left( \tfrac{1}{x_j}-\tfrac{\delta_j}{x_{-j}} \right) \left( \tfrac{1}{x_i}-\tfrac{\delta_i}{x_{-i}} \right) \left( \tfrac{1}{x_j}-\tfrac{\delta_j}{x_{-j}} \right) \\
    &= \tfrac{1}{x_ix_j^2} - \tfrac{1}{x_jx_i} \tfrac{\delta_j}{x_{-j}} - \tfrac{1}{x_j}\tfrac{\delta_i}{x_{-i}x_j} - \tfrac{\delta_j}{x_{-j}x_ix_j} + \tfrac{1}{x_j} \tfrac{\delta_i}{x_{-i}} \tfrac{\delta_j}{x_{-j}} + \tfrac{\delta_j}{x_{-j}x_i} \tfrac{\delta_j}{x_{-j}} + \tfrac{\delta_j}{x_{-j}}\tfrac{\delta_i}{x_{-i}x_j} - \tfrac{\delta_j}{x_{-j}} \tfrac{\delta_i}{x_{-i}} \tfrac{\delta_j}{x_{-j}} \\
    &= \tfrac{1}{x_ix_j^2} - \tfrac{\delta_j}{x_{i+j}x_{-j}^2} - \tfrac{\delta_i}{x_{i+j}x_{-i}x_j} - \tfrac{\delta_j}{x_{-j}x_ix_j} + \tfrac{\delta_{ij}}{x_ix_{-i-j}x_{-j}} + \tfrac{1}{x_{i+j}x_{-j}x_j} + \tfrac{\delta_{ji}}{x_{-i-j}x_{-i}x_j} - \tfrac{\delta_{jij}}{x_{-i-j}x_{-i}x_{-j}} \\
    &= \iDelta_j \left( \tfrac{1}{x_ix_j} + \tfrac{1}{x_{i+j}x_{-j}}\right) + \iDelta_i\tfrac{1}{x_{i+j}x_j} + \tfrac{\delta_{ij}}{x_ix_{-i-j}x_{-j}} + \tfrac{\delta_{ji}}{x_{-i-j}x_{-i}x_j} - \tfrac{\delta_{jij}}{x_{-i-j}x_{-i}x_{-j}} - \tfrac{1}{x_ix_{i+j}x_j}.
  \end{align*}
  Using the fact that $s_is_js_i=s_js_is_j$, we obtain
  \begin{equation} \label{eq:braid-coefficient}
    \iDelta_j\iDelta_i\iDelta_j-\iDelta_i\iDelta_j\iDelta_i = \iDelta_i \left( \tfrac{1}{x_{i+j}x_j} - \tfrac{1}{x_{i+j}x_{-i}} - \tfrac{1}{x_ix_j}\right) - \iDelta_j \left( \tfrac{1}{x_{i+j}x_i} - \tfrac{1}{x_{i+j}x_{-j}} - \tfrac{1}{x_ix_j}\right).
  \end{equation}

  The proofs of~\eqref{eq:Dem-BC-braid-relation} and~\eqref{eq:Dem-G2-braid-relation} are similar to the proof of~\eqref{eq:Dem-A-braid-relation} and will be omitted.
\end{proof}

\begin{rem}\label{rem:type-G}
  Note that Proposition~\ref{prop:FDA-braid-relations} asserts that all of the coefficients $\eta^{ji}_w$ lie in $\FA$ except in the case $m_{ij}=6$.  In fact, we expect that the coefficients $\xi_{ij}$ and $\xi_{ji}$ (for $i,j \in I$ with $m_{ij}=6$) also lie in $\FA$.  In other words, we conjecture that one in fact has $\eta_w^{ji} \in \FA$ (instead of $Q^F$) in the statement of Proposition~\ref{prop:FDA-braid-relations}.
\end{rem}

\begin{rem}\label{kijzero}
  Note that
  \[ \ts
    \kappa_{i,j} = \frac{x_i(x_{-i}-x_j)-x_{i+j}x_{-i}}{x_i x_{-i} x_j x_{i+j}}.
  \]
  By \cite[p.~809]{BE90}, the numerator of the above expression equals zero if and only if $F(u,v)=u+v+a_{11}uv$ for some $a_{11}\in R$ (i.e.\ if and only if $F$ is the additive or multiplicative {\fgl}). Therefore, contrary to the situation for the additive and multiplicative {\fgl}s, the formal Demazure elements do \emph{not} satisfy the braid relations in general (cf.\ \cite[Thm.~3.7]{BE90}).
\end{rem}

\begin{rem}
  Observe that the key difference between our setting and the setting of \cite{BE90} is that we deal with \emph{algebraic theories} for which the groups $\hh(BT)$ and their properties, which are used extensively in \cite{BE90}, are not well defined or remain unknown. Instead, we rely on the formal group algebra $\FA$ (as a replacement for $\hh(BT)$) and techniques introduced in \cite{CPZ09}.
\end{rem}

For each $w \in W$, fix a reduced decomposition $w = s_{i_1} \dotsm s_{i_k}$ and set
\begin{equation} \label{eq:iDelta_w-def}
  \iDelta_w = \iDelta_{i_1} \dotsm \iDelta_{i_k}.
\end{equation}
Note that, in general, $\iDelta_w$ depends on the choice of reduced decomposition.

\begin{defin}[$\tilde R$ and $R \llbracket \Lambda \rrbracket^\sim_F$]
  Let $\tilde R$ be the subalgebra of $Q^F$ defined by
  \begin{equation}
    \tilde R := R[W] \cdot R [\eta_{ij}^w\ |\ i,j \in I,\ w \in W,\ 1 \le \ell(w) \le m_{ij}-2],
  \end{equation}
  where $R[W] \cdot$ denotes the natural action of the group algebra $R[W]$ of $W$ on $Q^F$.  Similarly, define
  \begin{equation}
    R \llbracket \Lambda \rrbracket^\sim_F := R[W] \cdot R \llbracket \Lambda \rrbracket_F [\xi_{ij}, \xi_{ji}\ |\ i,j \in I,\ m_{ij}=6].
  \end{equation}
  Note that $R \llbracket \Lambda \rrbracket^\sim_F = R \llbracket \Lambda \rrbracket_F$ if $m_{ij} \ne 6$ for all $i,j \in I$.  In fact, we expect that $R \llbracket \Lambda \rrbracket^\sim_F = R \llbracket \Lambda \rrbracket_F$ in general (see Remark~\ref{rem:type-G}).
\end{defin}

The following lemma is an easy generalization of \cite[Thm.~4.6]{KoKu86} (which considers the case of the additive \fgl).

\begin{lem} \label{lem:iDelta-basis}
  The set $\{\iDelta_w\ |\ w \in W\}$ forms a basis of $D_F \otimes_R \tilde R$ as a right (or left) $\tilde R$-module and a basis of $\bD_F \otimes_{\FA} R \llbracket \Lambda \rrbracket^\sim_F$ as a right (or left) $R \llbracket \Lambda \rrbracket^\sim_F$-module.
\end{lem}

\begin{proof}
  Since $R$ is a domain, so are $\tilde R$ and $R \llbracket \Lambda \rrbracket^\sim_F$. By~\eqref{eq:Dem-elem-squared} and Proposition~\ref{prop:FDA-braid-relations}, we can write any product of formal Demazure elements as a $\tilde R$-linear combination of the elements $\iDelta_w$, $w \in W$.  Combined with Lemma~\ref{lem:Dcomm}, we can write any product of formal Demazure elements and elements of $R \llbracket \Lambda \rrbracket^\sim_F$ as an $ R \llbracket \Lambda \rrbracket^\sim_F$-linear combination of the elements $\iDelta_w$, $w \in W$.
  Thus $\{\iDelta_w\ |\ w \in W\}$ is a spanning set of the modules in the statement of the lemma.  Now, it is easy to see from the definition of the formal Demazure elements (Definition~\ref{def:Demazure-element}) that, for all $w \in W$,
  \[ \ts
    \iDelta_w = \sum_{v\, :\, \ell(v) \le \ell(w)} \delta_v a_w,
  \]
  where the sum is over elements $v \in W$ with length less than or equal to the length of $w$, $a_v \in Q^F$ for all $v$, and $a_w \ne 0$.  Thus, since $\{\delta_w\ |\ w \in W\}$ is a basis for $Q_W^F$ as a right (or left) $Q^F$-module, we see that $\{\iDelta_w\ |\ w \in W\}$ is also a basis for this module.  In particular, the set $\{\iDelta_w\ |\ w \in W\}$ is linearly independent over $Q^F$ and hence over $\tilde R$ or $R \llbracket \Lambda \rrbracket^\sim_F$.
\end{proof}

\begin{theo} \label{thm:Demazure-relations}
  Given a formal group law $(R,F)$, the formal affine Demazure algebra $\bD_{F}$ is generated as an $R$-algebra by $\FA$ and the formal Demazure elements $\iDelta_i$, $i \in I$, and satisfies the following relations:
  \begin{enumerate}
    \item \label{thm-item:Dem-affine-relation} $\varphi \iDelta_i = \iDelta_i s_i(\varphi)+ \Delta_{\alpha_i}(\varphi)$ for all $i\in I$ and $\varphi \in \FA$;

    \item $\iDelta_i^2=\iDelta_i \kappa_i$ for all $i\in I$, where $\kappa_i=\tfrac{1}{x_i}+\tfrac{1}{x_{-i}}\in \FA$;

    \item $\iDelta_i\iDelta_j=\iDelta_j\iDelta_i$ for all $i,j \in I$ such that $\left< \alpha_i^\vee, \alpha_j \right> = 0$;

    \item the braid relations of Proposition~\ref{prop:FDA-braid-relations} for all $i,j \in I$ such that $\left< \alpha_i^\vee, \alpha_j \right> \ne 0$.
  \end{enumerate}
  Furthermore, if $m_{ij} \ne 6$ for all $i,j \in I$, then the above form a complete set of relations.  In all cases, they form a complete set of relations (over $R \llbracket \Lambda \rrbracket^\sim_F$) for $\bD_F \otimes_{\FA} R \llbracket \Lambda \rrbracket^\sim_F$.
\end{theo}

\begin{proof}
  The first part of the theorem follows immediately from~\eqref{eq:Dem-elem-squared}, Lemma~\ref{lem:Dcomm} and Proposition~\ref{prop:FDA-braid-relations}.

  Since $R$ is a domain, so is $R \llbracket \Lambda \rrbracket^\sim_F$.  Let $\tilde \bD_F$ be the $R$-algebra generated by $R \llbracket \Lambda \rrbracket^\sim_F$ and elements $\iDelta_i'$, $i \in I$, subject to the relations given in the theorem.  Then we have a surjective ring homomorphism $\rho: \tilde \bD_F \to \bD_F \otimes_{\FA} R \llbracket \Lambda \rrbracket^\sim_F$ which is the identity on $R \llbracket \Lambda \rrbracket^\sim_F$ and maps $\iDelta_i'$ to $\iDelta_i$.  We wish to show that this map is an isomorphism.

  For $w \in W$, define $\iDelta_w'$ as in~\eqref{eq:iDelta_w-def}.  The relations among the $\iDelta'_w$ allow us to write any element of $\tilde \bD_F$ in the form
  \[ \ts
    \sum_{w \in W} \iDelta_w' a_w,\quad a_w \in R \llbracket \Lambda \rrbracket^\sim_F.
  \]
  Suppose such an element is in the kernel of $\rho$.  Then
  \[ \ts
    0 = \rho \left( \sum_{w \in W} \iDelta_w' a_w \right) = \sum_{w \in W} \iDelta_w a_w.
  \]
  By Lemma~\ref{lem:iDelta-basis}, this implies that $a_w=0$ for all $w$.  Thus $\rho$ is injective and hence an isomorphism.  This completes the proof of the proposition once we recall that $R \llbracket \Lambda \rrbracket^\sim_F\, \cong \FA$ in simply laced type.
\end{proof}

\begin{rem}
  We expect that the relations of Theorem~\ref{thm:Demazure-relations} in fact form a complete set of relations for $\bD_F$ in all cases (see Remark~\ref{rem:type-G}).
\end{rem}

%
\section{Formal (affine) Demazure algebras: examples and further properties} \label{sec:FDAs-egs}
%

In this section we specialize the definition of the formal (affine) Demazure algebra to various {\fgl}s.  We then prove several important facts about these algebras in general.  The first proposition demonstrates that our definition recovers classical objects in the additive and multiplicative cases.

\begin{prop} \label{prop:FDA-add-mult}
  \begin{enumerate}
    \item \label{prop-item:FADA-additive} For the additive \fgl\ over $R=\Z$, the formal affine Demazure algebra $\bD_A$ is isomorphic to the completion of the \emph{nil Hecke ring} of \cite[Def.~4.12]{KoKu86}.  In this case, all the relations among the $\iDelta_i$ are given by the braid relations and the nilpotence relations $\iDelta_i^2=0$.

    \item \label{prop-item:FDA-additive} For the additive \fgl\ over $R=\C$, the formal Demazure algebra $D_A$ is isomorphic to the completion of the nil Hecke ring of \cite[Def.~3]{EvBr87}.

    \item \label{prop-item:FDA-mult} For the multiplicative \fgl\ over $R=\mathbb{C}$ with $\beta=1$, the formal Demazure algebra $D_M$ is the completion of the \emph{0-Hecke algebra}, which is the classical Hecke algebra specialized at $q=0$.  In this case all the relations among the $\iDelta_i$ are given by the braid relations and the idempotence relations.
  \end{enumerate}
\end{prop}

\begin{proof}
  For the additive \fgl\ over $R=\Z$, $Q^A$ is a subring of the field of fractions of the ring $S^*(\Lw)^\wedge \cong R \llbracket \Lambda \rrbracket_A$ and $\iDelta_i$ is the $-x_i$ of \cite[($\text{I}_{24}$)]{KoKu86}.  This proves part~\eqref{prop-item:FADA-additive}.  Similarly, if $R=\C$, then our $\iDelta_i$ corresponds to the $X_i$ of \cite{EvBr87}, proving part~\eqref{prop-item:FDA-additive}.

  For the multiplicative \fgl\ over $R=\C$ with $\beta=1$, we have that $Q^M$ is a subring of the field of  fractions of $\C[\Lw]^\wedge$ and $-\iDelta_{\alpha_i}$ is the $B_i$ of \cite[\S2]{EvBr87} if we identify the simple root $\alpha_i$ of \cite{EvBr87} with our $-\alpha_i$.  This proves part~\eqref{prop-item:FDA-mult}.
\end{proof}

We now consider some other {\fgl}s, where our definition appears to give new algebras.

\begin{example}[Lorentz affine Demazure algebra]
  Consider the Lorentz \fgl\ $(R,F_L)$. Then $x_{i+j} = \tfrac{x_i +x_j}{1+\beta x_ix_j}$, $\beta\in R$.  Since $x_{-\lambda}=-x_\lambda$ for all $\lambda \in \Lw$, we have
  \begin{equation} \label{eq:kappa-ij-Lorentz}
    \kappa_{i,j}=\tfrac{1+\beta x_i x_j}{x_i+x_j}\cdot
    \tfrac{x_i+x_j}{x_ix_j}-\tfrac{1}{x_ix_j}=\beta
  \end{equation}
  for $i,j \in I$ with $\left< \alpha_i^\vee, \alpha_j \right>=-1$.  Therefore, relation~\eqref{eq:Dem-A-braid-relation} becomes
  \[
    \iDelta_j\iDelta_i\iDelta_j-\iDelta_i\iDelta_j\iDelta_i=\beta(\iDelta_i-\iDelta_j)\;\text{ for all }i,j\text{ such that }\left< \alpha_i^\vee, \alpha_j \right>=-1.
  \]
\end{example}

\begin{example}[Elliptic affine Demazure algebra]
  Consider the elliptic \fgl\ $(R,F_E)$.  Set $\mu_i = \mu_E(x_i)$ and $g_{ij}^E = g^E(x_i,x_j)$.  Then, for $i,j \in I$ with $\langle \alpha_i^\vee, \alpha_j \rangle = -1$, we have
  \begin{multline*}
    \kappa_{i,j} = \tfrac{x_i ( x_{-i} - x_j) - x_{i+j} x_{-i}}{x_i x_{-i} x_j x_{i+j}}
    = \tfrac{x_{-i} - x_j + x_{i+j} \mu_i}{x_{-i} x_j x_{i+j}}
    = \tfrac{x_{-i} - x_j + \mu_i ( x_i + x_j - x_i x_j g^E_{ij})}{x_{-i} x_j x_{i+j}}
    = \tfrac{- x_j + \mu_i x_j - \mu_i x_i x_j g^E_{ij}} {x_{-i} x_j x_{i+j}} \\
    = \tfrac{\mu_i - 1 - \mu_i x_i g^E_{ij}}{x_{-i} x_{i+j}}
    = \tfrac{\mu_{-i}^{-1} - 1 - \mu_i x_i g^E_{ij}}{x_{-i} x_{i+j}}
    = \tfrac{- a_1 x_{-i} - a_3 v(x_{-i}) + x_{-i} g^E_{ij}}{x_{i+j} x_{-i}}
    = \tfrac{g^E_{ij} - a_1}{x_{i+j}} - \tfrac{a_3 v(x_{-i})}{x_{i+j} x_{-i}}.
\end{multline*}
\end{example}

\begin{theo} \label{thm:Demazure-algebra-isom}
  For any two formal group laws $(R,F)$ and $(R,F')$ over the same ring $R$, we have $\bD_{R_\Q,F} \cong \bD_{R_\Q,F'}$ as algebras, where $R_\Q = R \otimes_\Z \Q$.
\end{theo}

\begin{proof}
  It suffices to prove the result for the special case where $F'=F_A$.  There is an isomorphism of {\fgl}s
  \[
    e_F : (R_\Q,F_A) \to (R_\Q,F)
  \]
  given by the exponential series $e_F (u) \in R_\Q \llbracket u \rrbracket$ (see Section~\ref{sec:formal-group-rings}).  This induces an isomorphism of formal group algebras
  \[
    e_F^\star : R_\Q \llbracket \Lambda \rrbracket_F \to R_\Q \llbracket \Lambda \rrbracket_A.
  \]
  This map commutes with the action of $W$ and thus we have an induced isomorphism of twisted formal group algebras
  \[
    e_F^\star : Q_W^{(R_\Q,F)} \to Q_W^{(R_\Q,A)}.
  \]
  Thus $\bD_{R_\Q,F}$ is isomorphic to its image $D' := e_F^\star(\bD_{R_\Q,F})$ under this map. Now, $D'$ is generated over $R_\Q \llbracket \Lambda \rrbracket_A$ by the elements
  \[
    e_F^\star(\Delta_i^F) = \tfrac{1}{e_F(x_i)} (1-\delta_i) = \tfrac{x_i}{e_F(x_i)} \Delta_i^A,\quad i \in I.
  \]
  Since $e_F(x_i)/x_i \in R_\Q \llbracket \Lambda \rrbracket_A$ is
  invertible in $R_\Q \llbracket \Lambda \rrbracket_A$ (because its constant term is invertible in $R_\Q$ -- see Example~\ref{eg:exp-series}\eqref{eg-item:exp-general}), $D'$ is also generated over $R_\Q \llbracket \Lambda \rrbracket_A$ by $\Delta_i^A$, $i \in I$, and thus isomorphic to $\bD_{R_\Q,A}$.
\end{proof}

\begin{rem} \label{rem:isom-Dem-alg}
  Note that while Theorem~\ref{thm:Demazure-algebra-isom} shows that all affine Demazure algebras are isomorphic when the coefficient ring is $R_\Q$, the isomorphism is \emph{not} the naive one sending $\iDelta_i^F$ to $\iDelta_i^{F'}$.  Furthermore, the completion (with respect to the augmentation map) is crucial.  No assertion is made regarding an isomorphism (even over $R_\Q$) of \emph{truncated} versions.
\end{rem}

%
\section{Formal (affine) Hecke algebras: definitions} \label{sec:FHAs}
%

In the present section we define the formal (affine) Hecke algebras.  Our goal is to construct generalizations of the usual affine Hecke algebra and its degenerate analogue, which are deformations of the affine 0-Hecke algebra and affine nil Hecke algebra respectively.  These two examples will correspond to the multiplicative periodic and additive \fgl\ cases of our more general construction.  We begin by reminding the reader of the definition of these classical objects.

\begin{defin}[Hecke algebra] \label{def:Hecke-algebra}
  The (classical) \emph{Hecke algebra} $H$ associated to the Weyl group $W$ is the $\Z[t,t^{-1}]$-algebra with 1 generated (as a $\Z[t,t^{-1}]$-algebra) by elements $T_i := T_{s_i}$, $i \in I$, modulo
  \begin{enumerate}
    \item \label{def-item:classical-Hecke-quadratic} the quadratic relations $(T_i+t^{-1})(T_i - t)=0$ for all $i \in I$, and
    \item the braid relations $T_i T_j T_i \dotsm = T_j T_i T_j \dotsm$ ($m_{ij}$ factors in both products) for any $i \ne j$ in $I$ with $s_i s_j$ of order $m_{ij}$ in $W$.
  \end{enumerate}
\end{defin}

\begin{defin}[Affine Hecke algebra] \label{def:affine-Hecke-algebra}
  The (classical) \emph{affine Hecke algebra} $\bH$ is $H \otimes_{\Z[t,t^{-1}]} \Z[t,t^{-1}][\Lambda]$, where the factors $H$ and $\Z[t,t^{-1}][\Lambda]$ are subalgebras, and the relations between the two factors are given by
  \[ \ts
    e^\lambda T_i - T_i e^{s_i(\lambda)} = (t-t^{-1}) \frac{e^\lambda - e^{s_i(\lambda)}}{1 - e^{-\alpha_i}},\quad \lambda \in \Lambda,\ i \in I.
  \]
\end{defin}

\begin{rem} \label{rem:Hecke-conventions}
  In Defintions~\ref{def:Hecke-algebra} and~\ref{def:affine-Hecke-algebra}, we have followed the conventions found, for instance, in \cite[pp.~71--72]{CMHL02} (except that we use $t$ in place of $v$ and $T_i$ in place of $\tilde T_i$).  These conventions differ somewhat from those found in other places in the literature.  For instance, $H$ as defined above is isomorphic to $H' \otimes_{\Z[q,q^{-1}]} \Z[q^{1/2},q^{-1/2}]$, where $H'$ is the Hecke algebra as defined in \cite[\S7.4]{Hum90} or \cite[Def.~7.1.1]{CG10}.  The $T_{s_i}$ appearing in \cite{Hum90,CG10} correspond to $tT_{s_i}$ in our notation, where $t$ corresponds to $q^{1/2}$.
\end{rem}

\begin{defin}[Degenerate affine Hecke algebra]
  Let $\epsilon$ be an indeterminate.  The \emph{degenerate affine Hecke algebra} $\bH_\mathrm{deg}$ is the unital associative $\Z[\epsilon]$-algebra that is $\Z[W] \otimes_\Z S_{\Z[\epsilon]}^*(\Lambda)$ as a $\Z[\epsilon]$-module and such that the subspaces $\Z[W]$ and $S_{\Z[\epsilon]}^*(\Lambda)$ are subalgebras and the following relations hold:
  \[
    \delta_i \lambda - s_i(\lambda) \delta_i = - \epsilon \langle \alpha_i^\vee, \lambda \rangle,\quad i \in I,\ \lambda \in \Lambda.
  \]
\end{defin}

Fix a free abelian group $\Gamma$ of rank $1$ with generator $\gamma$.  Denote by $R_F$ the formal group algebra $R \llbracket \Gamma \rrbracket_F$.  For instance, $\Z_M \cong \Z[t,t^{-1}]^\wedge$ and $\Z_A \cong
\Z\llbracket \gamma\rrbracket$ (see Example~\ref{eg:FGL-torus}).  Let $Q' := Q^{(R_F,F)}$ denote the subring of the fraction field of $R_F\llbracket \Lw \rrbracket_F$ generated by $R_F \llbracket \Lw \rrbracket_F$, $\{x_\lambda^{-1}\ |\ \lambda \in \Lambda \setminus \{0\}\}$ and $\{(\kappa^F_\alpha)^{-1}\ |\ \alpha \in \Phi,\ \kappa^F_\alpha \ne 0\}$.  Let $Q_W' \cong R_F[W] \ltimes_R Q'$ be the respective twisted formal group algebra over $R_F$ (see Definition~\ref{def:twisted-FGA}).  We will continue to use the shorthand~\eqref{eq:subscript-shorthand}.  We are now ready to define our second main objects of study.

\begin{defin}[Formal (affine) Hecke algebra] \label{def:formal-HA}
  The \emph{formal Hecke algebra} $H_{F}$ is the $R_{F}$-subalgebra of $Q_W'$
  generated by the elements
  \begin{equation} \label{eq:T_i-def}
    T_i^F :=
    \begin{cases}
      \iDelta_i^F \frac{\Theta_F}{\kappa_i^F} + \delta_i \mu_F(x_\gamma) & \text{if }\kappa^F \neq 0, \\
      2\iDelta_i^F x_\gamma + \delta_i & \text{if }\kappa^F=0,
    \end{cases}
  \end{equation}
  for all $i \in I$, where $\Theta_F := \mu_F(x_\gamma)-\mu_F(x_{-\gamma})\in R_F$.  The \emph{formal affine Hecke algebra} $\bH_{F}$ is the $R_F$-subalgebra of $Q_W'$ generated by $H_{F}$ and
  \[
    R_F \llbracket \Lambda \rrbracket_F^\kappa :=
    \begin{cases}
      R_F \llbracket \Lambda \rrbracket_F[(\kappa^F_\alpha)^{-1}\ |\ \alpha \in \Phi] & \text{if } \kappa^F \ne 0, \\
      R_F \llbracket \Lambda \rrbracket_F & \text{if } \kappa^F = 0.
    \end{cases}
  \]
  We sometimes write $T_i$ when the \fgl\ is understood from the context.  When we wish to specify the coefficient ring, we write $H_{R,F}$ (resp. $\bH_{R,F})$ for $H_F$ (resp.\ $\bH_F$).
\end{defin}

\begin{rem}
  \begin{enumerate}
    \item If $a_{11}$ (in the notation of~\eqref{equ-fgl}) is invertible in $R$, then $\kappa_\alpha^F$ is invertible in $R_F \llbracket \Lambda \rrbracket_F$ for all $\alpha \in \Phi$.  Thus, $R_F \llbracket \Lambda \rrbracket_F^\kappa = R_F \llbracket \Lambda \rrbracket_F$.

    \item The coefficients $\tfrac{\Theta_F}{\kappa_i}$, $\mu_F(x_\gamma)$, and $x_\gamma - x_{-\gamma}$ appearing in Definition~\ref{def:formal-HA} are all invariant under the action of $s_i$.

    \item In the multiplicative case, we have
        \[
          \tfrac{\Theta_F}{\kappa_i} = \tfrac{\beta(x_\gamma - x_{-\gamma})}{\beta} = x_\gamma - x_{-\gamma}.
        \]
        Since the additive \fgl\ is the $\beta\to 0$ limit of the multiplicative, this motivates the choice of coefficient of $\iDelta_i^F$ in the case $\kappa^F=0$.  More generally, one can show that when $\frac{\Theta_F}{\kappa_i^F}$ is expanded as a power series in $x_i$ and $x_\gamma$, the leading term is equal to $2 x_\gamma$.  Similarly, when $\kappa^F=0$, we have $x_{-\gamma}=-x_{\gamma}$ and so $\mu_F(x_\gamma)=1$, the coefficient of $\delta_i$.
  \end{enumerate}
\end{rem}

\begin{lem} \label{lem:affine-Hecke-rel}
  For all $\psi \in Q'$ and $i \in I$, we have
  \begin{equation} \label{eq:formal-affine-hecke-Tx-relation}
    \psi T_i  - T_i s_i(\psi) =
    \begin{cases}
      \frac{\Theta_F}{\kappa_i} \Delta_{\alpha_i}^F(\psi) & \text{if } \kappa^F \ne 0, \\
      2x_\gamma \Delta_{\alpha_i}^F(\psi) & \text{if } \kappa^F = 0.
    \end{cases}
  \end{equation}
  In particular, $\varphi T_i - T_i s_i(\varphi) \in R_F \llbracket \Lambda \rrbracket_F^\kappa$ for all $\varphi \in R_F \llbracket \Lambda \rrbracket_F^\kappa$.
\end{lem}

\begin{proof}
  Let $a$ and $b$ be the coefficients of $\iDelta_i$ and $\delta_i$ in~\eqref{eq:T_i-def}, so that $T_i = \iDelta_i a + \delta_i b$.  By Lemma~\ref{lem:Dcomm}, we have
  \[
    \psi T_i = \psi (\iDelta_i a + \delta_{s_i} b)=(\iDelta_i s_i(\psi) + \Delta_{\alpha_i} (\psi)) a + \delta_{s_i} s_i(\psi) b = T_i s_i(\psi) +a\Delta_{\alpha_i}(\psi).
  \]
  The last statement is an easy verification left to the reader.
\end{proof}

\begin{lem} \label{lem:T_i-quadratic}
  The elements $T_i$, $i \in I$, satisfy the quadratic relation
  \begin{equation} \label{eq:quadrelat1}
    T_i^2=T_i\Theta_F +1.
  \end{equation}
  Thus $T_i$ is invertible and $T_i^{-1} = T_i -\Theta_F$.  Furthermore
  \begin{equation} \label{eq:quadrelat2}
    (T_i+\mu_F(x_{-\gamma}))(T_i-\mu_F(x_\gamma))=0.
  \end{equation}
\end{lem}

\begin{proof}
  Let $a$ and $b$ be the coefficients of $\iDelta_i$ and $\delta_i$ in~\eqref{eq:T_i-def}, so that $T_i = \iDelta_i a + \delta_i b$.  Using~\eqref{eq:Dem-elem-squared} and the fact that $\iDelta_i \delta_i + \delta_i \iDelta_i = (\delta_i-1) \kappa_i^F$, we have
  \[
    T_i^2 = \iDelta_i^2 a^2 + (\iDelta_i \delta_i+\delta_i \iDelta_i)ab+b^2 = \iDelta_i \kappa_i^F a^2+(\delta_i-1) \kappa_i^F ab +b^2 = T_i (\kappa_i^F a) +b(b-\kappa_i^F a).
  \]
  One readily verifies that in both cases in~\eqref{eq:T_i-def}, we have $\kappa_i^F a=\Theta_F$ and $b(b-\kappa_i^F a)=1$, completing the proof of the first statement in the lemma.  The second two statements follow by direct computation.
\end{proof}

\begin{rem}
  Because of~\eqref{eq:quadrelat2} and the fact that $\mu_F(x_{-\gamma}) = \mu_F(x_\gamma)^{-1}$, one may think of the power series $\mu_F(x_\gamma)$ as a generalization of the deformation parameter $t$ of the classical Hecke algebra (see Definition~\ref{def:Hecke-algebra}\eqref{def-item:classical-Hecke-quadratic}).
\end{rem}

\begin{prop} \label{prop:FHA-braid-relations}
  Suppose $i,j \in I$ and let $m_{ij}$ be the order of $s_i s_j$ in $W$.  Then
  \begin{equation} \label{eq:T_i-general-braid-relation}
    \underbrace{T_j T_i T_j \dotsm}_{m_{ij} \text{ terms }} - \underbrace{T_i T_j T_i \dotsm}_{m_{ij} \text{ terms}} = \sum_{w \in W,\, \ell(w) \le m_{ij}-2} T_w \tau_w^{ji}
  \end{equation}
  for some $\tau^{ij}_w \in Q'$.  In particular, we have the following.
  \begin{enumerate}
    \item \label{prop-item:T_i-nonadjacent} If $\langle \alpha_i^\vee, \alpha_j \rangle=0$, so that $m_{ij}=2$, then $T_i T_j = T_j T_i$.

    \item \label{prop-item:T_i-adjacent} If $\langle \alpha_i^\vee, \alpha_j \rangle = \langle \alpha_j^\vee, \alpha_i \rangle=-1$ so that $m_{ij}=3$, then
        \begin{equation} \label{eq:T_i-braid-relation}
          T_j T_i T_j - T_i T_j T_i = (T_i - T_j) \sigma_{ij},\quad \sigma_{ij} = \chi_{i + j} (\chi_j - \chi_{-i}) - \chi_i \chi_j,
        \end{equation}
        where
        \[
          \chi_\alpha =
          \begin{cases}
            \frac{\Theta_F}{x_\alpha \kappa_\alpha} & \text{if } \kappa^F \ne 0, \\
            \frac{2 x_\gamma}{x_\alpha} & \text{if } \kappa^F = 0,
          \end{cases}
            \qquad \alpha \in \Phi.
        \]
        (We use the usual convention that $\chi_{\pm i} = \chi_{\pm \alpha_i}$ and $\chi_{\pm i \pm j} = \chi_{\pm \alpha_i \pm \alpha_j}$.)  Moreover, $\sigma_{ij} = \sigma_{ji}$ commutes with $\delta_i$ and $\delta_j$ (and hence with $T_i$ and $T_j$).  If $\kappa^F=0$, then $\sigma_{ij} = 4 x_\gamma^2 \kappa_{i,j} \in R_F \llbracket \Lambda \rrbracket_F$.
  \end{enumerate}
\end{prop}

\begin{proof}
  Set $\mu = \mu_F(x_\gamma)$ (thus $\mu=1$ iff $\kappa^F=0$ by Lemma~\ref{lem:kappa-eq-zeo}) and $\chi_\alpha' = \mu - \chi_\alpha$ for all $\alpha \in \Phi$.  In both cases (i.e.\ $\kappa^F \ne 0$ or $\kappa^F=0$), we have $T_j = \chi_j + \chi_j' \delta_j$.  Then the first part of the proposition follows from the proof of Proposition~\ref{prop:FDA-braid-relations}.

  Part~\eqref{prop-item:T_i-nonadjacent} follows immediately from Proposition~\ref{prop:FDA-braid-relations}\eqref{prop-item:Delta_i-nonadjacent} and the facts that, under the assumptions, $\delta_i \delta_j = \delta_j \delta_i$, $\iDelta_i \delta_j = \delta_j \iDelta_i$, and $\iDelta_j \delta_i = \delta_i \iDelta_j$.

  It remains to prove~\eqref{prop-item:T_i-adjacent}.  We have
  \begin{align*}
    T_j T_i T_j &= \big( \chi_j + (\mu-\chi_j) \delta_j \big) \big( \chi_i + (\mu - \chi_i) \delta_i \big) \big( \chi_j + (\mu-\chi_j) \delta_j \big) \\
    &= \chi_i \chi_j^2 + (\mu - \chi_j)(\mu - \chi_{-j})\chi_{i+j} + \big( (\mu - \chi_j)\chi_{i+j} \chi_{-j} + (\mu - \chi_j) \chi_i \chi_j \big) \delta_j \\
    & \qquad + (\mu-\chi_i) \chi_j \chi_{i+j} \delta_i + (\mu-\chi_j)(\mu-\chi_{i+j}) \chi_i \delta_j \delta_i + (\mu-\chi_i)(\mu-\chi_{i+j}) \chi_j \delta_i \delta_j \\
    & \qquad + (\mu - \chi_j)(\mu - \chi_i)(\mu-\chi_{i+j}) \delta_j \delta_i \delta_j.
  \end{align*}
  Using the fact that $\delta_j \delta_i \delta_j = \delta_j \delta_i \delta_j$, we see that
  \begin{align*}
    T_j T_i T_j - T_i T_j T_i &= \chi_i \chi_j^2 - \chi_i^2 \chi_j + (\mu - \chi_j)(\mu - \chi_{-j})\chi_{i+j} - (\mu - \chi_i)(\mu - \chi_{-i}) \chi_{i+j} \\
    &\qquad \qquad \qquad \qquad \qquad \qquad \qquad \qquad - \sigma_{ji} (\mu - \chi_j) \delta_j + \sigma_{ij} (\mu -\chi_i) \delta_i \\
    &= \sigma_{ij} \chi_i - \sigma_{ji} \chi_j - \sigma_{ji} (\mu - \chi_j) \delta_j + \sigma_{ij} (\mu -\chi_i) \delta_i + \mu \chi_{i+j} (\chi_i + \chi_{-i} - \chi_j - \chi_{-j}) \\
    &= \sigma_{ij} T_i - \sigma_{ji} T_j + \mu \chi_{i+j} (\chi_i + \chi_{-i} - \chi_j - \chi_{-j}) \\
    &= \sigma_{ij} T_i - \sigma_{ji} T_j + \mu(\sigma_{ji} - \sigma_{ij}) \\
    &= \sigma_{ij} (T_i - \mu) - \sigma_{ji} (T_j - \mu).
  \end{align*}

  If $\kappa^F=0$, then clearly $\sigma_{ij} = 4x_\gamma^2 \kappa_{i,j}$.  Since, in this case, $\kappa_{i,j}=\kappa_{j,i}$ (use the fact that $x_{-i}=-x_i$ in \eqref{eq:kappa_ij}), we have $\sigma_{ij}=\sigma_{ji}$.  If $\kappa^F\ne 0$, we have
  \begin{align*} \ts
    \sigma_{ij} &= \ts \Theta_F^2 \frac{x_i x_{-i} \kappa_i - x_ix_j\kappa_j - x_{-i}x_{i+j}\kappa_{i+j}}{x_i x_{-i} x_j x_{i+j} \kappa_i \kappa_j \kappa_{i+j}} \\
    &= \ts \Theta_F^2
    \frac{(x_i + x_{-i})x_{-j}x_{-i-j} - (x_j+x_{-j})x_ix_{-i-j} - (x_{i+j} + x_{-i-j})x_{-i}x_{-j}}{(x_i + x_{-i})(x_j + x_{-j})(x_{i+j} + x_{-i-j})} \\
    &= \ts - \Theta_F^2 \frac{x_ix_jx_{-i-j} + x_{-i}x_{-j}x_{i+j}}{(x_i + x_{-i})(x_j + x_{-j})(x_{i+j} + x_{-i-j})} = - \Theta_F^2 \left( \tfrac{1}{x_ix_jx_{-i-j}} + \tfrac{1}{x_{-i}x_{-j}x_{i+j}} \right) \tfrac{1}{\kappa_i\kappa_j\kappa_{i+j}},
  \end{align*}
  which implies that $\sigma_{ij}=\sigma_{ji}$.  Thus we have
  \[
    T_j T_i T_j - T_i T_j T_i=\sigma_{ij}(T_i-T_j).
  \]
  The fact that $\sigma_{ij}$ commutes with $\delta_i$ and $\delta_j$ is an direct verification left to the reader.
\end{proof}

For each $w \in W$, fix a reduced decomposition $w = s_{i_1} \dotsm s_{i_k}$ and set
\begin{equation} \label{eq:T_w-def}
  T_w = T_{i_1} \dotsm T_{i_k}.
\end{equation}
Note that, in general, $T_w$ depends on the choice of reduced decomposition.

\begin{defin}[$\tilde R_F$ and $R_F \llbracket \Lambda \rrbracket^\sim_F$]
  Let $\tilde R_F$ be the subalgebra of $Q'$ defined by
  \begin{equation}
    \tilde R_F := R_F[W] \cdot R_F [\tau^{ji}_w\ |\ i,j \in I,\ w \in W,\ \ell(w) \le m_{ij}-2].
  \end{equation}
  where $R_F[W] \cdot$ denotes the natural action of the group algebra $R_F[W]$ of $W$ on $Q'$.  Similarly, define
  \begin{equation}
    R_F \llbracket \Lambda \rrbracket^\sim_F := R_F[W] \cdot R_F \llbracket \Lambda \rrbracket_F^\kappa [\tau^{ji}_w\ |\ i,j \in I,\ w \in W,\ \ell(w) \le m_{ij}-2].
  \end{equation}
  Note that $R_F \llbracket \Lambda \rrbracket^\sim_F = R_F \llbracket \Lambda \rrbracket_F$ if the root system is simply laced and $\kappa^F = 0$ (since $\sigma_{ij} = 4 x_\gamma^2 \kappa_{i,j} \in R_F \llbracket \Lambda \rrbracket_F$ in that case).
\end{defin}

\begin{rem}
  In fact, we expect that the coefficients $\tau^{ji}_w$ lie in $R_F \llbracket \Lambda \rrbracket_F$. In this case, we would have $R_F \llbracket \Lambda \rrbracket_F^\sim = R_F \llbracket \Lambda \rrbracket_F^\kappa$.
\end{rem}

\begin{lem} \label{lem:T_w-basis}
  The set $\{T_w\ |\ w \in W\}$ forms a basis of $H_F \otimes_{R_F} \tilde R_F$ as a right (or left) $\tilde R_F$-module and a basis of $\bH_F \otimes_{R_F \llbracket \Lambda \rrbracket_F} R_F \llbracket \Lambda \rrbracket^\sim_F$ as a right (or left) $R_F \llbracket \Lambda \rrbracket_F^\sim$-module.
\end{lem}

\begin{proof}
  The proof is analogous to that of Lemma~\ref{lem:iDelta-basis} and will be omitted.
\end{proof}

\begin{theo} \label{thm:formal-Hecke-relations}
  Given a formal group law $(R,F)$, the formal affine Hecke algebra $\bH_{F}$ is generated as an $R_F$-algebra by $R_F \llbracket \Lambda \rrbracket_F^\kappa$ and the elements $T_i$, $i \in I$, and satisfies
  \begin{enumerate}
    \item \label{thm-itm:T_i-affine-relation} relation~\eqref{eq:formal-affine-hecke-Tx-relation} for all $i \in I$ and $\varphi \in R_F \llbracket \Lambda \rrbracket_F^\kappa$,

    \item \label{thm-item:T_i-quadratic} $(T_i+\mu_F(x_{-\gamma}))(T_i-\mu_F(x_\gamma))=0$ for all $i \in I$,

    \item $T_i T_j = T_j T_i$ for all $i,j \in I$ such that $\left< \alpha_i^\vee, \alpha_j \right> = 0$,

    \item \label{thm-item:T_i-adjacent} relation~\eqref{eq:T_i-braid-relation} for all $i,j \in I$ such that $m_{ij}=3$, and

    \item \label{thm-item:T_i-higher-order-braid} relation~\eqref{eq:T_i-general-braid-relation} for all $i,j \in I$ such that $m_{ij} > 3$.
  \end{enumerate}
  Furthermore, \eqref{thm-itm:T_i-affine-relation}--\eqref{thm-item:T_i-higher-order-braid} form a complete set of relations (over $R_F \llbracket \Lambda \rrbracket^\sim_F$) for $\bH_F \otimes_{R_F \llbracket \Lambda \rrbracket_F} R_F \llbracket \Lambda \rrbracket^\sim_F$.
\end{theo}

\begin{proof}
  The first part of the theorem follows immediately from Lemmas~\ref{lem:affine-Hecke-rel} and~\ref{lem:T_i-quadratic} and Proposition~\ref{prop:FHA-braid-relations}.  The second part is analogous to the proof of Theorem~\ref{thm:Demazure-relations} and will be omitted.
\end{proof}

%
\section{Formal (affine) Hecke algebras: examples and further properties} \label{sec:FHAs-egs}
%

In this final section we specialize the definition of the formal (affine) Hecke algebra to various {\fgl}s, yielding classical algebras as well as ones that seem to be new.  We then prove several important facts about these algebras in general.

There is a natural action of $Q_W'$ on $Q'$ where $Q' \subseteq Q_W'$ acts by left multiplication and $R_F[W]$ acts via the action of the Weyl group.  Thus we have a map $Q_W' \to \End_{R_F} Q'$.  Since the operators $T^F_i$ preserve $R_F \llbracket \Lambda \rrbracket_F^\kappa$, we have an induced map $\bH_F \to \End_{R_F} R_F \llbracket \Lambda \rrbracket_F^\kappa$ of $R_F$-algebras.  Recall that if $a_{11}$ (in the notation of~\eqref{equ-fgl}) is invertible in $R$, then $\kappa_\alpha^F$ is invertible for all $\alpha \in \Phi$, and so $R_F \llbracket \Lambda \rrbracket_F^\kappa = R_F \llbracket \Lambda \rrbracket_F$.

\begin{prop} \label{prop:Hecke-alg-embedding}
  If $a_{11}$ is invertible in $R$, then the map $\bH_F \to \End_{R_F} R_F \llbracket \Lambda \rrbracket_F$ described above is injective.  In other words, the natural action of $\bH_F$ on $R_F \llbracket \Lambda \rrbracket_F$ is faithful.
\end{prop}

\begin{proof}
Suppose, contrary to the statement of the proposition, that the given map is not injective.  Let $a \in \bH_F$ be in the kernel of this map, with $a \ne 0$.  In other words, $a$ acts by zero on $R_F \llbracket \Lambda \rrbracket_F$ under the associated action.  By Lemma~\ref{lem:T_w-basis}, we may write
  \[ \ts
    a = \sum_{w \in W} T_w a_w,\quad a_w \in R_F \llbracket \Lambda \rrbracket_F^\sim.
  \]
  Now, clearly $a \varphi$ also acts by zero on $R_F \llbracket \Lambda \rrbracket_F$ for all $\varphi \in R_F \llbracket \Lambda \rrbracket_F$.  Choosing $\varphi$ to be a common denominator of all the $a_w$, we see that we may assume that $a_w \in R_F \llbracket \Lambda \rrbracket_F$ for all $w \in W$.

  For $\varphi \in R_F \llbracket \Lambda \rrbracket_F$, define the \emph{degree} of $\varphi$ to be
  \[
    \deg \varphi := \max \{m \in \Z_{\ge 0}\ |\ \varphi \in \mathcal{I}_F^m\},
  \]
  where $\mathcal{I}_F$ is the kernel of the augmentation map $\varepsilon \colon R_F \llbracket \Lambda \rrbracket_F \to R_F$ (i.e.\ the element $x_\gamma$ is \emph{not} mapped to zero).  We adopt the convention that $\deg 0 = -\infty$.  Then the formal Demazure operators lower degree and the coefficients $\mu_F(x_\gamma)$, $\frac{\Theta_F}{\kappa_i^F}$ and $x_\gamma$ appearing in Definition~\eqref{eq:T_i-def} of $T_i$ preserve degree.  Thus, if $\deg \varphi = m$, we have
  \[
    T_i(\varphi) = \mu_F(x_\gamma) s_i(\varphi) + (\text{terms of degree} < m).
  \]
  Furthermore, $\deg (\varphi \varphi') = \deg \varphi + \deg \varphi'$ for $\varphi, \varphi' \in R_F \llbracket \Lambda \rrbracket_F$.  Indeed, it follows by definition that $\deg (\varphi \varphi') \ge \deg \varphi + \deg \varphi'$.  If $\deg (\varphi \varphi') > \deg \varphi + \deg \varphi'$, then in the associated graded algebra we have $\varphi\varphi'=0$, where $\varphi\neq 0$ and $\varphi'\neq 0$.  Identifying the associated graded algebra with the polynomial algebra (by \cite[Lem.~4.2]{CPZ09}) we obtain a contradiction as the polynomial algebra is a domain.

  Let $m$ be the maximum degree of the $a_w$, $w \in W$, and set $W' = \{w \in W\ |\ \deg a_w = m\}$. Then, for all $\varphi \in R_F \llbracket \Lambda \rrbracket_F$, we have
  \begin{align*} \ts
    0 = a(\varphi) &= \ts \sum_{w \in W'} T_w (a_w \varphi) + \sum_{w \in W \setminus W'} T_w (a_w \varphi) \\
    &= \ts \sum_{w \in W'} \mu_F(x_\gamma)^{\ell(w)} s_w(a_w \varphi) + b,
  \end{align*}
  where the last summation lies in $\mathcal{I}_F^{m + \deg \varphi}$ and $b \not \in \mathcal{I}_F^{m + \deg \varphi}$.  It follows that
  \[ \ts
    \sum_{w \in W'} \mu_F(x_\gamma)^{\ell(w)} s_w(a_w \varphi) = \sum_{w \in W'} s_w \left( \mu_F(x_\gamma)^{\ell(w)} a_w \varphi \right) = 0 \quad \text{for all } \varphi \in R_F \llbracket \Lambda \rrbracket_F.
  \]
  The above sum is therefore also equal to zero in $\mathcal{I}_F^{m+\deg \varphi}/\mathcal{I}_F^{m+\deg{\varphi}+1}$.  But $\bigoplus_n \mathcal{I}_F^n/\mathcal{I}_F^{n+1} \cong S_{R_F}^*(\Lambda)$, by~\cite[Lem.~4.2]{CPZ09}.  Since the action of $R_F \llbracket W \rrbracket \ltimes S_{R_F}^*(\Lambda)$ on $S_{R_F}^*(\Lambda)$ is faithful (see, for example, the argument in~\cite[Second Proof of Thm.~3.2.2]{Kle05}), we have that $\mu_F(x_\gamma)^{\ell(w)} a_w = 0$ (hence $a_w=0$) for all $w \in W'$.  But this contradicts the choice of $m$.
\end{proof}

\begin{rem}
  In the additive and multiplicative cases, Proposition~\ref{prop:Hecke-alg-embedding} reduces to known embeddings of the (degenerate) affine Hecke algebra into endomorphism rings.  See the proof of Proposition~\ref{prop:FHA-classical}.
\end{rem}

The following proposition demonstrates that our definition of the formal (affine) Hecke algebra recovers classical objects in the additive and multiplicative cases.

\begin{prop} \label{prop:FHA-classical}
  Suppose $R=\Z$.
  \begin{enumerate}
    \item \label{prop-item:additive-FHA} For the additive \fgl, we have the following isomorphisms of algebras:
    \[
      \bH_A \cong \bH_\mathrm{deg}^\wedge,\quad H_A \cong \Z_A[W] \cong \Z[W] \otimes_\Z \Z \llbracket \gamma \rrbracket,
    \]
    where $\epsilon=-2\gamma$ and $\bH_\mathrm{deg}^\wedge$ denotes the completion of $\bH_\mathrm{deg}$ with respect to the kernel of the augmentation map $\Z[W]\otimes_{\Z}S^*_{\Z[\epsilon]}(\Lambda)\to \Z[W]$ given by $\epsilon\mapsto 0$ and $\lambda \mapsto 0$, $\lambda \in \Lambda$.

    \item \label{prop-item:mult-FHA} For the multiplicative periodic \fgl, we have the following isomorphisms of algebras:
    \[
      \bH_M \cong \bH^\wedge,\quad H_M \cong H \otimes_{\Z[t,t^{-1}]} \Z[ t,t^{-1}]^\wedge,
    \]
    where $\bH^\wedge$ denotes the completion of $\bH^\wedge$ with respect to the kernel of the map $H\otimes_{\Z[t,t^{-1}]} \Z[t,t^{-1}][\Lambda] \to \Z[W]$ that maps $t \mapsto 1$ and $e^\lambda \mapsto 1$, $\lambda\in \Lambda$.
  \end{enumerate}
\end{prop}

\begin{proof}
  It is easy to see that for the additive and multiplicative {\fgl}s in simply laced type, the relations of Theorem~\ref{thm:formal-Hecke-relations} become the relations of the respective algebras in the statement of the proposition.  However, we provide a proof that remains valid in all types (i.e.\ not necessarily simply laced).  Note that in the both the additive case (where $\kappa^F=0$) and multiplicative periodic case (where $a_{11} = \beta$ is invertible and hence all $\kappa_\alpha$, $\alpha \in \Phi$, are invertible), we have $R_F \llbracket \Lambda \rrbracket_F^\kappa = R_F \llbracket \Lambda \rrbracket_F$.

  Consider first the additive \fgl.  Recall the identification $\Z_A \cong \Z \llbracket \gamma \rrbracket$ of Example~\ref{eg:FGL-torus}.  The injective map $\bH_A \hookrightarrow \End_{\Z_A} \Z_A \llbracket \Lambda \rrbracket_A$ is given on the $T_i$ by
  \[ \ts
    T_i = 2 \gamma \iDelta^A_i + \delta_i \mapsto s_i + 2 \gamma \frac{1}{\alpha} ( s_i - 1 ).
  \]
  Thus $\bH_A$ is isomorphic to the subalgebra $\bH_A'$ of $\End_{\Z_A} \Z_A \llbracket \Lambda \rrbracket_A$ generated by multiplication by elements of $\Z_A \llbracket \Lambda \rrbracket_A$ and the operators $s_i + 2 \gamma \frac{1}{\alpha} (s_i-1)$.

  Observe that, in the notation of \cite[\S12]{Ginzburg98}, the algebra $S_{\Z_A}^*(\Lambda)^\wedge \otimes_\Z \C =S_\Z^* (\Lambda)^\wedge \otimes_\Z \C \llbracket \gamma\rrbracket$ can be identified with the completion of the algebra $\C[\mathfrak{h},\gamma]$ of polynomial functions on $\mathfrak{h}$ with coefficients in $\C[\gamma]$.
  If we let $\epsilon = -2\gamma$, then we see that $\bH_A'$ is precisely the completion of the image of $\bH_\mathrm{deg}$ under the faithful action on $\C[\mathfrak{h},\epsilon]$ given by Demazure-Lusztig type operators (see \cite[Prop.~12.2]{Ginzburg98} or \cite[Second Proof of Thm.~3.2.2]{Kle05}).  This proves the first isomorphism of part~\eqref{prop-item:additive-FHA}.  The second follows by considering the subalgebra generated by the $T_i$.

  Now consider the multiplicative periodic \fgl\ $F_M(u,v) = u + v - \beta uv$, $\beta \in \Z^\times$.
  We have (see Example~\ref{FGL:muexamples}\eqref{example-item:multiplicative-FGL})
  \[
    \mu_M(x_\gamma)=1-\beta x_{-\gamma}=t \text{ and }\Theta_{M}=\beta (x_\gamma - x_{-\gamma}) = t-t^{-1}\in \Z[t,t^{-1}]^\wedge,
  \]
  under the identifications of Example~\ref{eg:FGL-torus}.  Using the above and the identifications of Example~\ref{eg:FGAs}\eqref{eg-item:FGA-mult}, the injective map $\bH_M \hookrightarrow \End_{\Z_M} \Z_M \llbracket \Lambda \rrbracket_M$ is given on the $T_i$ by
  \begin{equation} \label{eq:mult-Lusztig-Demazure-operator} \ts
    T_i = \iDelta^M_i \frac{\Theta_M}{\kappa_i^M} + \delta_i \mu_M(x_\gamma) = \frac{t-t^{-1}}{1 - e^{-\alpha_i}}(1-\delta_i) + t \delta_i \mapsto t^{-1} \frac{1-s_i}{e^{-\alpha_i} - 1} - t \frac{1 - e^{-\alpha_i} s_i}{e^{-\alpha_i} - 1}.
  \end{equation}
  We identify $\Z_M \llbracket \Lambda \rrbracket_M$ with $\Z[q,q^{-1}][P]^\wedge$ (completion with respect to the kernel of the map $\Z[q,q^{-1}][P] \to \Z[P]$ sending $q \mapsto 1$) in the notation of \cite{Lus85} (where the $P$ and $q$ of \cite{Lus85} are our $\Lambda$ and $t^2$, respectively) via the map $e^\lambda \mapsto -\lambda$ (see Example~\ref{eg:FGAs}\eqref{eg-item:FGA-mult}).  (The negative sign in front of the $\alpha_i$ arises from the twisting of the action of $\Z[q,q^{-1}][P]$ on itself by a sign in \cite[(8.2)]{Lus85}.)  Under this identification, the right hand side of~\eqref{eq:mult-Lusztig-Demazure-operator} corresponds to the Demazure-Lusztig operator \cite[(8.1)]{Lus85}, where the $T_s$ of \cite{Lus85} corresponds to our $tT_i$, where $s=s_i$ (see Remark~\ref{rem:Hecke-conventions}).  Therefore, the actions of $\bH_M$ and $\bH$ on $\Z_M \llbracket \Lambda \rrbracket_M \cong \Z[q,q^{-1}][P]^\wedge$ coincide.  The action of $\bH_M$ is faithful by Proposition~\ref{prop:Hecke-alg-embedding} and the action of $\bH$ is also known to be faithful (see, for example, \cite[Prop.~12.2(i)]{Ginzburg98} or note that the action of $\bH$ specializes to the standard action of $\Z[W] \ltimes \Z[\Lambda]$ on $\Z[\Lambda]$ when $q=1$).  Thus we have the first isomorphism of part~\eqref{prop-item:mult-FHA}.  The second follows by considering the subalgebra generated by the $T_i$.
\end{proof}

For other {\fgl}s our definition seems to give new algebras as the following examples indicate.

\begin{example}[Lorentz case]
  For the Lorentz \fgl\ $F_L$, we have $\mu_L(u)=1$, $\Theta_L=0$, and $\kappa=0$.  Since $\kappa_{i,j} = \beta$ (see~\eqref{eq:kappa-ij-Lorentz}), we have $\sigma_{ij} = 4 \beta x_\gamma^2$.  Thus the relations \eqref{thm-itm:T_i-affine-relation}--\eqref{thm-item:T_i-adjacent} of Theorem~\ref{thm:formal-Hecke-relations} become
  \begin{enumerate}
    \item $\varphi T_i - T_i s_i(\varphi) = 2x_\gamma \Delta_{\alpha_i}^L(\varphi)$ for all $\varphi \in R_F \llbracket \Lambda \rrbracket_F,\ i \in I$.
    \item $T_i^2 = 1$ for all $i \in I$,
    \item $T_i T_j = T_j T_i$ for all $i,j \in I$ such that $\langle \alpha_i^\vee, \alpha_j \rangle=0$,
    \item $T_i T_j T_i - T_j T_i T_j = 4 \beta x_\gamma^2 (T_i - T_j)$ for all $i,j \in I$ such that  $\langle \alpha_i^\vee, \alpha_j \rangle=-1$.
  \end{enumerate}
  These form a complete set of relations in the simply laced case.
\end{example}

\begin{example}[Elliptic case]
  For the elliptic \fgl\ $F_E$, we have
  \[ \ts
    \mu_E(u) = \frac{1}{1 - a_1u - a_3 v(u)},\quad \Theta_E = \frac{2\psi - \psi^2}{1-\psi} = \frac{\psi}{1-\psi} + \psi,
  \]
  where $\psi = a_1 x_\gamma + a_3 v(x_\gamma)$ (see Example~\ref{FGL:muexamples}\eqref{eg-item:elliptic-FGL}). If, for example, $a_3=0$, then
  \[ \ts
    \mu_E(u)=\frac{1}{1-a_1u},\quad \Theta_E = \frac{2a_1 x_\gamma - a_1^2 x_\gamma^2}{1 - a_1 x_\gamma},\quad \kappa_i=a_1\quad \text{for all } i \in I,
  \]
  and so
  \[ \ts
    T_i = \Delta^E_i \frac{2x_\gamma - a_1 x_\gamma^2}{1-a_1 x_\gamma} + \delta_i \frac{1}{1-a_1u} \quad \text{for all } i \in I.
  \]
  Furthermore, when $a_3=0$, we have
  \[ \ts
    \chi_i = \frac{\Theta_E}{x_i a_1}
  \]
  and so
  \[ \ts
    \sigma_{ij} = \frac{\Theta_E}{x_i a_1} \frac{\Theta_E}{x_j a_1} + \frac{\Theta_E}{x_{i+j}a_1} \left( \frac{\Theta_E}{x_{-i} a_1} - \frac{\Theta_E}{x_ja_1} \right) = - \frac{\Theta_E^2}{a_1^2} \kappa_{i,j}.
  \]
\end{example}

\begin{example}[Universal formal Hecke algebra]
  We call the formal Hecke algebra $H_{U}$ corresponding to the universal \fgl\ $F_U$ the \emph{universal formal Hecke algebra}.  Observe that $H_{U}$ is an algebra over $\LL_{U}$, where $\LL$ is the Lazard ring.  Note that in this case we have
  \begin{align*}
    \Theta_{U} &= -a_{11}(x_{\gamma}-x_{-\gamma})+a_{11}^2(x_{\gamma}^2-x_{-\gamma}^2) - (a_{11}^3+a_{12}a_{11}-a_{22}+2a_{13})(x_\gamma^3 - x_{-\gamma}^3)+ \dotsb \\
    &=-2a_{11}x_\gamma - 2(a_{11}^3 + a_{11}a_{12} - a_{22} + 2a_{13})x_\gamma^3+\dotsb.
  \end{align*}
\end{example}

\begin{theo} \label{thm:Hecke-alg-isom}
  Suppose $(R,F)$ and $(R,F')$ are {\fgl}s over the same ring $R$, with either $\kappa^F=0$ or $a_{11}$ invertible in $R$ (in the notation of~\eqref{equ-fgl}).  Then
  \[
    \bH_F \otimes_{R_F \llbracket \Lambda \rrbracket_F} R'_F \llbracket \Lambda \rrbracket_F \cong \bH_{F'} \otimes_{R_{F'} \llbracket \Lambda \rrbracket_{F'}} R'_{F'} \llbracket \Lambda \rrbracket_{F'}
  \]
  as algebras, where $R'_F = (R \otimes_\Z \Q)_F \otimes_\Q \Q[x_\gamma^{-1}]$ (and similarly, with $F$ replaced by $F'$).
\end{theo}

\begin{proof}
  It suffices to prove the result when $F'=F_A$.  Let $R_\Q = R \otimes_\Z \Q$.  As in the proof of Theorem~\ref{thm:Demazure-algebra-isom}, we have an isomorphism of twisted formal group algebras
  \[
    e_F^\star : Q_W^{(R'_F,F)} \to Q_W^{(R'_A,A)}.
  \]
  Since
  \[ \ts
    e_F^\star(x_\gamma^{-1}) = \frac{1}{e_F(x_\gamma)} = \frac{x_\gamma}{e_F(x_\gamma)} x_\gamma^{-1}
  \]
  and $\frac{x_\gamma}{e_F(x_\gamma)} \in (R_\Q)_A$ is invertible in $(R_\Q)_A$, we see that $e_F^\star (R'_F) = R'_A$ and so $e_F^\star (R_F' \llbracket \Lambda \rrbracket_F) = R_A' \llbracket \Lambda \rrbracket_A$.
  The algebra $\bH_F \otimes_{R_F \llbracket \Lambda \rrbracket_F} R'_F \llbracket \Lambda \rrbracket_F$ is isomorphic to its image $\bH' := e_F^\star (\bH_F \otimes_{R_F \llbracket \Lambda \rrbracket_F} R'_F \llbracket \Lambda \rrbracket_F)$ under $e_F^\star$.

  We first consider the case where $\kappa^F = 0$.  Then $\bH_F \otimes_{R_F \llbracket \Lambda \rrbracket_F} R'_F \llbracket \Lambda \rrbracket_F$ is generated over $R'_F \llbracket \Lambda \rrbracket_F$ by (the element 1 and) the elements
  \[
    T_i^F - 1 = 2 \Delta_i^F x_\gamma + \delta_i - 1 = \Xi_i^F (\delta_i-1),\quad i \in I,
  \]
  where
  \[ \ts
    \Xi_i^F = 1 - \frac{2x_\gamma}{x_i} \in Q_W^{(R_F,F)}.
  \]
  We see that
  \[ \ts
    e_F^\star(T_i^F-1) = e_F^\star(\Xi_i^F)(\delta_i - 1) = \frac{e_F^\star(\Xi_i^F)}{\Xi_i^A} (T_i^A-1).
  \]
  Thus it suffices to show that $e_F^\star(\Xi_i^F)/\Xi_i^A$ lies in $R'_A \llbracket \Lambda \rrbracket_A$ and is invertible in $R'_A \llbracket \Lambda \rrbracket_A$ (i.e.\ has invertible constant term).  Now,
  \begin{align*} \ts
    \frac{e_F^\star(\Xi_i^F)}{\Xi_i^A} &= \ts \left( 1 - \frac{2 e_F(x_\gamma)}{e_F(x_i)} \right) \left(1 - \frac{2x_\gamma}{x_i} \right)^{-1} \\
    &= \ts \left( \frac{x_i}{e_F(x_i)} \frac{e_F(x_\gamma)}{x_\gamma} - \frac{x_i}{2x_\gamma} \right) \left(1 - \frac{x_i}{2x_\gamma} \right)^{-1}.
  \end{align*}
  Note that $e_F(x_i)/x_i, e_F(x_\gamma)/x_\gamma \in R'_A \llbracket \Lambda \rrbracket_A$ are invertible in $R'_A \llbracket \Lambda \rrbracket_A$ (with constant term one).  Since $1-x_i/(2x_\gamma) \in R'_A \llbracket \Lambda \rrbracket_A$ is also clearly invertible in $R'_A \llbracket \Lambda \rrbracket_A$, we are done.

  Now consider the case where $\kappa^F \ne 0$ and $a_{11}$ is invertible in $R$ (hence in $R_A'$).  The elements
  \[
    T_i^F - \mu_F(x_\gamma) = \rho_i(\delta_i-1),\quad i \in I,
  \]
  where
  \[ \ts
    \rho_i = \mu_F(x_\gamma) - \frac{\Theta_F}{x_i \kappa_i^F}\in Q_W^{(R_F,F)}
  \]
  generate $\bH_F \otimes_{R_F \llbracket \Lambda \rrbracket_F} R'_F \llbracket \Lambda \rrbracket_F$ over $R'_F \llbracket \Lambda \rrbracket_F$ (along with the element 1).  Since
  \[ \ts
    e_F^\star (T_i^F - \mu_F(x_\gamma)) = \frac{e_F^\star(\rho_i)}{\Xi_i^A}(T_i^A-1),
  \]
  it follows as in the $\kappa^F=0$ case that it suffices to show that $e_F^\star(\rho_i)/\Xi_i^A$ lies in $R'_A \llbracket \Lambda \rrbracket_A$ and is invertible in $R'_A \llbracket \Lambda \rrbracket_A$ (i.e.\ has invertible constant term).

  For any $x\in R'_A \llbracket \Lambda \rrbracket_A$ we set
  \[ \ts
    \psi(x)=\frac{1-\mu_F(e_F(x))}{x}=a_{11}+O(1) \in R'_A \llbracket \Lambda \rrbracket_A
  \]
  so that $\mu_F(e_F(x))=1-x\psi(x)$. Then (as $x_{-\lambda}=-x_{\lambda}$ in $R_A \llbracket \Lambda \rrbracket_A$)
  \begin{align*} \ts
    \frac{e_F^\star(\rho_i)}{\Xi_i^A} &= \ts \left( \mu_F(e_F(x_\gamma)) - \frac{\mu_F(e_F(x_\gamma)) - \mu_F(e_F(-x_{\gamma}))}{1-\mu_F(e_F(x_{-i}))} \right) \left( 1 - \frac{2x_\gamma}{x_i} \right)^{-1} \\
    &= \ts \left(  1-x_\gamma \psi(x_\gamma) - \frac{x_\gamma}{x_i}\cdot      \frac{\psi(-x_\gamma) + \psi(x_\gamma)}{\psi(-x_i)}\right) \left( 1 - \frac{2x_\gamma}{x_i} \right)^{-1} \\
    &= \ts \left( \frac{\psi(-x_\gamma) + \psi(x_\gamma)}{\psi(-x_i)}+x_i\psi(x_\gamma) - \frac{x_i}{x_\gamma}\right) \left( 2 - \frac{x_i}{x_\gamma} \right)^{-1}.
  \end{align*}
  Since $a_{11}$ is invertible, we have
  \[ \ts
    \frac{\psi(-x_\gamma)+\psi(x_\gamma)}{\psi(-x_i)}=\frac{2a_{11}+O(1)}{a_{11}+O(1)}=2+O(1).
  \]

  Combining all of the above computations, we see that
  \[ \ts
    \frac{e_F^\star(\rho_i)}{\Xi_i^A} = 1 + O(1) \in R_A' \llbracket \Lambda \rrbracket_A
  \]
  is invertible in $R_A' \llbracket \Lambda \rrbracket_A$ as desired.
\end{proof}

\begin{rem}
  \begin{asparaenum}
    \item It is known that certain localizations or completions of the affine Hecke algebra and degenerate affine Hecke algebra are isomorphic (see \cite[Thm.~9.3]{Lus89} and \cite[\S3.1.7]{Rou08}).  Theorem~\ref{thm:Hecke-alg-isom} can be seen as an analogue of these results.

    \item Note that while Theorem~\ref{thm:Hecke-alg-isom} shows that all affine Hecke algebras (satisfying the hypotheses of the proposition) become isomorphic over appropriate rings, the isomorphism is \emph{not} the naive one sending $T_i^F$ to $T_i^{F'}$.  Furthermore, the completion (with respect to the augmentation map) is crucial.  No assertion is made regarding an isomorphism (even over $\Q$) of \emph{truncated} versions.  See Remark~\ref{rem:isom-Dem-alg}.
    \end{asparaenum}
\end{rem}


\bibliographystyle{alpha}
\bibliography{HMSZ-biblist}

\begin{thebibliography}{CMHL02}

\bibitem[BE90]{BE90}
Paul Bressler and Sam Evens.
\newblock The {S}chubert calculus, braid relations, and generalized cohomology.
\newblock {\em Trans. Amer. Math. Soc.}, 317(2):799--811, 1990.

\bibitem[Bou81]{Bou81}
Nicolas Bourbaki.
\newblock {\em \'{E}l\'ements de math\'ematique}.
\newblock Masson, Paris, 1981.
\newblock Groupes et alg{\`e}bres de Lie. Chapitres 4, 5 et 6. [Lie groups and
  Lie algebras. Chapters 4, 5 and 6].

\bibitem[CG10]{CG10}
Neil Chriss and Victor Ginzburg.
\newblock {\em Representation theory and complex geometry}.
\newblock Modern Birkh\"auser Classics. Birkh\"auser Boston Inc., Boston, MA,
  2010.
\newblock Reprint of the 1997 edition.

\bibitem[CMHL02]{CMHL02}
Ivan Cherednik, Yavor Markov, Roger Howe, and George Lusztig.
\newblock {\em Iwahori-{H}ecke algebras and their representation theory},
  volume 1804 of {\em Lecture Notes in Mathematics}.
\newblock Springer-Verlag, Berlin, 2002.
\newblock Lectures from the C.I.M.E. Summer School held in Martina-Franca, June
  28--July 6, 1999, Edited by M. Welleda Baldoni and Dan Barbasch.

\bibitem[CPZ13]{CPZ09}
Baptiste Calm{\`e}s, Victor Petrov, and Kirill Zainoulline.
\newblock Invariants, torsion indices and oriented cohomology of complete
  flags.
\newblock {\em Ann. Sci. \'Ec. Norm. Sup\'er. (4)}, 46(3), 2013.
\newblock Preprint available at arXiv:0905.1341v2 [math.AG].

\bibitem[CZZ]{CZZ09}
Baptiste Calm\'es, Kirill Zainoulline, and Changlong Zhong.
\newblock A coproduct structure on the formal affine demazure algebra.
\newblock arXiv:arXiv:1209.1676 [math.RA].

\bibitem[Dem73]{Dem73}
Michel Demazure.
\newblock Invariants sym\'etriques entiers des groupes de {W}eyl et torsion.
\newblock {\em Invent. Math.}, 21:287--301, 1973.

\bibitem[Dem74]{dem-des}
Michel Demazure.
\newblock D\'esingularisation des vari\'et\'es de {S}chubert
  g\'en\'eralis\'ees.
\newblock {\em Ann. Sci. \'Ecole Norm. Sup. (4)}, 7:53--88, 1974.
\newblock Collection of articles dedicated to Henri Cartan on the occasion of
  his 70th birthday, I.

\bibitem[EB87]{EvBr87}
Sam Evens and Paul Bressler.
\newblock On certain {H}ecke rings.
\newblock {\em Proc. Nat. Acad. Sci. U.S.A.}, 84(3):624--625, 1987.

\bibitem[Fr{\"o}68]{Fro68}
Albrecht Fr{\"o}hlich.
\newblock {\em Formal groups}.
\newblock Lecture Notes in Mathematics, No. 74. Springer-Verlag, Berlin, 1968.

\bibitem[Gin]{Ginzburg98}
Victor Ginzburg.
\newblock Geometric methods in the representation theory of {H}ecke algebras
  and quantum groups (notes by {V}.~{B}aranovsky).
\newblock arXiv:math/9802004v3 [math.AG].

\bibitem[GKV]{GKV95}
Victor Ginzburg, Mikhail Kapranov, and Eric Vasserot.
\newblock Elliptic algebras and equivariant elliptic cohomology.
\newblock arXiv:q-alg/9505012.

\bibitem[GKV97]{GKV97}
Victor Ginzburg, Mikhail Kapranov, and Eric Vasserot.
\newblock Residue construction of {H}ecke algebras.
\newblock {\em Adv. Math.}, 128(1):1--19, 1997.

\bibitem[GZ12]{GZ12}
Stefan Gille and Kirill Zainoulline.
\newblock Equivariant pretheories and invariants of torsors.
\newblock {\em Transform. Groups}, 17(2):471--498, 2012.

\bibitem[Hum90]{Hum90}
James~E. Humphreys.
\newblock {\em Reflection groups and {C}oxeter groups}, volume~29 of {\em
  Cambridge Studies in Advanced Mathematics}.
\newblock Cambridge University Press, Cambridge, 1990.

\bibitem[Kac90]{Kac90}
Victor~G. Kac.
\newblock {\em Infinite-dimensional {L}ie algebras}.
\newblock Cambridge University Press, Cambridge, third edition, 1990.

\bibitem[KK86]{KoKu86}
Bertram Kostant and Shrawan Kumar.
\newblock The nil {H}ecke ring and cohomology of {$G/P$} for a {K}ac-{M}oody
  group {$G$}.
\newblock {\em Adv. in Math.}, 62(3):187--237, 1986.

\bibitem[Kle05]{Kle05}
Alexander Kleshchev.
\newblock {\em Linear and projective representations of symmetric groups},
  volume 163 of {\em Cambridge Tracts in Mathematics}.
\newblock Cambridge University Press, Cambridge, 2005.

\bibitem[Lan87]{Lan87}
Serge Lang.
\newblock {\em Elliptic functions}, volume 112 of {\em Graduate Texts in
  Mathematics}.
\newblock Springer-Verlag, New York, second edition, 1987.
\newblock With an appendix by J. Tate.

\bibitem[LM07]{LM07}
Mark Levine and Fabien Morel.
\newblock {\em Algebraic cobordism}.
\newblock Springer Monographs in Mathematics. Springer, Berlin, 2007.

\bibitem[Lus85]{Lus85}
George Lusztig.
\newblock Equivariant {$K$}-theory and representations of {H}ecke algebras.
\newblock {\em Proc. Amer. Math. Soc.}, 94(2):337--342, 1985.

\bibitem[Lus89]{Lus89}
George Lusztig.
\newblock Affine {H}ecke algebras and their graded version.
\newblock {\em J. Amer. Math. Soc.}, 2(3):599--635, 1989.

\bibitem[Pan03]{Pan03}
Ivan Panin.
\newblock Oriented cohomology theories of algebraic varieties.
\newblock {\em $K$-Theory}, 30(3):265--314, 2003.
\newblock Special issue in honor of Hyman Bass on his seventieth birthday. Part
  III.

\bibitem[PR99]{PiRa99}
Harsh Pittie and Arun Ram.
\newblock A {P}ieri-{C}hevalley formula in the {$K$}-theory of a
  {$G/B$}-bundle.
\newblock {\em Electron. Res. Announc. Amer. Math. Soc.}, 5:102--107, 1999.

\bibitem[Rou]{Rou08}
Rapha\"el Rouquier.
\newblock 2-{K}ac-{M}oody algebras.
\newblock arXiv:math/0812.5023v1 [math.RT].

\bibitem[Sil09]{Sil09}
Joseph~H. Silverman.
\newblock {\em The arithmetic of elliptic curves}, volume 106 of {\em Graduate
  Texts in Mathematics}.
\newblock Springer, Dordrecht, second edition, 2009.

\bibitem[Tat74]{Tate}
John~T. Tate.
\newblock The arithmetic of elliptic curves.
\newblock {\em Invent. Math.}, 23:179--206, 1974.

\end{thebibliography}

\end{document}